\documentclass[12pt,a4paper,oneside,reqno]{article}

\usepackage[utf8]{inputenc}
\usepackage[english]{babel}
\usepackage[left=1in,right=1in,top=1in,bottom=1in]{geometry}
\usepackage{amsfonts,amssymb,amsmath,amsthm}
\usepackage[symbol]{footmisc}
\usepackage{longtable,array}
\newcolumntype{L}[1]{>{\raggedright\arraybackslash$}p{#1}<{$}}
\usepackage{enumerate}
\usepackage{cite}
\usepackage{xcolor}

\usepackage{hyperref}
\hypersetup{hidelinks}
\usepackage{fancyhdr}
\theoremstyle{plain}
\newtheorem{theorem}{Theorem}

\newtheorem{proposition}[theorem]{Proposition}
\newtheorem{corollary}[theorem]{Corollary}
\theoremstyle{definition}

\makeatletter
\@removefromreset{theorem}{section}\@removefromreset{equation}{section}

\makeatother
\usepackage{adjustbox}
\usepackage{tikz}\usetikzlibrary{arrows}

\DeclareMathOperator{\Aut}{Aut}
\DeclareMathOperator{\tr}{tr}

\DeclareMathOperator{\GL}{GL}

\IfFileExists{fouriernc.sty}{\usepackage{fouriernc}}{}

\newcommand{\doi}[1]{\href{https://doi.org/#1}{doi:\detokenize{#1}}}
\newcommand{\arxiv}[1]{\href{https://arxiv.org/abs/#1}{arXiv:#1}}

\usepackage{mathrsfs,cancel,calrsfs}
\usepackage[all]{xy}
\usepackage{tikz-cd}
\usepackage{blkarray, bigstrut}
\DeclareMathAlphabet{\pazocal}{OMS}{zplm}{m}{n}
\newcommand{\bigzero}{\mbox{\normalfont\Large\bfseries 0}}
\newcommand{\rvline}{\hspace*{-\arraycolsep}\vline\hspace*{-\arraycolsep}}
\newcommand{\mi}{\mathfrak{m}}

\theoremstyle{plain}
\newtheorem{Th}[theorem]{Theorem}
\newtheorem{Lem}[theorem]{Lemma}
\newtheorem{prop}[theorem]{Proposition}
\theoremstyle{definition}
\newtheorem{Remark}[theorem]{Remark}

\newcommand{\F}{{\mathbb{F}}}
\newcommand{\Fs}{{\tiny\mathbb{F}}}

\renewcommand{\Aut}{{\mathrm{Aut}}}
\renewcommand{\tr}{\mathop{\rm Tr}}
\newcommand{\tor}{\mathop{\rm Tor}}

\newcommand{\Sa}{S}

\newcommand{\M}{\mathcal{M}}
\newcommand{\Endo}{\mathop{\text{End}}}
\newcommand{\rad}{\mathop{\text{rad}}}

\def\esc#1{\langle#1\rangle}
\def\sc{strongly connected}
\def\remove#1{}
\def\aut{\mathop{\hbox{aut}}}
\def\Aff{\mathop{\hbox{\rm Aff}}}
\def\affaut{\mathop{\hbox{\bf aut}}}
\def\alg{\mathop{\hbox{\rm Alg}}}
\def\grp{\mathop{\hbox{Grp}}}
\def\ch{\mathop{\hbox{\rm char}}}
\def\a{\alpha}

\def\b{\beta}
\def\g{\gamma}
\def\d{\delta}
\def\l{\lambda}
\def\m{\mu}
\def\tree{\mathop{\hbox{\bf tree}}}
\def\lie{\mathop{\hbox{\bf lie}}}
\def\GL{\mathop{\hbox{\bf GL}}}

\def\gl{\mathop{\hbox{\rm gl}}}
\def\aff{\mathop{\mathfrak{aff}}}
\def\ll{\mathcal{\ell}}

\begin{document}

\noindent\fbox{\parbox{\dimexpr\textwidth-2\fboxsep-2\fboxrule\relax}{\small
\textbf{Corrected version (verified by Claude).} This is a corrected version of the paper I.~Kaygorodov, D.~Mart\'in Barquero, and C.~Mart\'in Gonz\'alez, \emph{Conservative algebras of $2$-dimensional algebras, V}, Siberian Electronic Mathematical Reports, 22 (2025), no.~1, 587--620, \doi{10.33048/semi.2025.22.039}.}}

\bigskip
\bigskip
\begin{center}
{\Large Conservative algebras of $2$-dimensional algebras, V}

\bigskip
\bigskip
{\bf Ivan Kaygorodov, Dolores Mart\'in Barquero, and C\'andido Mart\'in Gonz\'alez}
\end{center}

\bigskip\bigskip

\noindent{\small\textbf{Abstract.} The notion of conservative algebras appeared in a paper of Kantor in 1972. Later, he defined the conservative algebra $W(n)$ of all algebras (i.e. bilinear maps) on the $n$-dimensional vector space. If $n>1$, then the algebra $W(n)$ does not belong to any well-known class of algebras (such as associative, Lie, Jordan, or Leibniz algebras). It looks like $W(n)$ plays a role in the theory of conservative algebras similar to the role of $\mathfrak{gl}_n$ in the theory of Lie algebras. Namely, an arbitrary conservative algebra can be obtained from a universal algebra $W(n)$ for some $n \in \mathbb{N}.$ The present paper is part of a series of papers dedicated to the study of the algebra $W(2)$ and its principal subalgebras.}

\bigskip
\bigskip
\noindent {\bf Keywords:}
{\it bilinear map, conservative algebra, graph of an algebra, multiplication algebra, automorphism group scheme.}

\bigskip
\bigskip
\noindent {\bf MSC2020:} 17A30 (primary); 05C20, 14L15, 17A36, 17A99 (secondary).

\bigskip
\bigskip
\tableofcontents

\bigskip

\newpage 
\section*{Introduction}
\addcontentsline{toc}{section}{Introduction}

We will use the notation $R$ for a (commutative) ring with unit. Also if $M$ is an $R$-module and $n\ge 2$ an integer, we write
$\tor_n(M):=\{x\in M\colon n x=0\}$.
The algebras under consideration in this work are not necessarily
unital or associative.
A multiplication on a vector space $W$ is a bilinear mapping $W\times W \to W$. We denote by $(W,P)$ the algebra with underlying space $W$ and multiplication $P$. Given a vector space $W$, a linear mapping $A:W\rightarrow W$, and a bilinear mapping $B:W\times W\to W$, we can define a multiplication $[ A,B ] :W\times W\to W$ by the formula
$$[ A,B ] (x,y)=A(B(x,y))-B(A(x),y)-B(x,A(y))$$
for $x,y\in W$. For an algebra $A$ with a multiplication $P$ and $x\in A$ we denote by $L_x^P$ the operator of left multiplication by $x$. If the multiplication $P$ is fixed, we write $L_x$ instead of $L_x^P$.

In 1990 Kantor~\cite{Kantor90} defined the multiplication $\cdot$ on the set of all algebras (i.e. all multiplications) on the $n$-dimensional vector space $V_n$ as follows:
$$A\cdot B = [L_{e}^A,B],$$
where $A$ and $B$ are multiplications and $e\in V_n$ is some fixed vector.
Let $W(n)$  denote the algebra of all algebra structures on $V_n$ with multiplication defined above.
If $n>1$, then the algebra $W(n)$ does not belong to any well-known class of algebras (such as associative, Lie, Jordan, or Leibniz algebras). The algebra $W(n)$ turns out to be a conservative algebra (see below).

In 1972 Kantor~\cite{Kantor72} introduced conservative algebras as a generalization of Jordan algebras
 (also, see surveys about the study of conservative algebras and superalgebras \cite{pp20,k23}). Namely, an algebra $A=(W,P)$ is called a {\it conservative algebra} if there is a new multiplication $F:W\times W\rightarrow W$ such that
\begin{eqnarray*}\label{tojd_oper}
[L_b^P, [L_a^P,P]]=-[L_{F(a,b)}^P,P]
\end{eqnarray*}
for all $a,b\in W$. In other words, the following identity holds for all $a,b,x,y\in W$:
\begin{multline*}\label{tojdestvo_glavnoe}
b(a(xy)-(ax)y-x(ay))-a((bx)y)+(a(bx))y+(bx)(ay)\\
-a(x(by))+(ax)(by)+x(a(by))
=-F(a,b)(xy)+(F(a,b)x)y+x(F(a,b)y).
\end{multline*}
The algebra $(W,F)$ is called an algebra {\it associated} to $A$.
The main subclass of conservative algebras is the variety of terminal algebras,
which defined by the conservative identity with
$F(a,b)=\frac{1}{3}(2 ab+ ba).$
It includes the varieties of Leibniz and Jordan algebras as subvarieties.

Let us recall some well-known results about conservative algebras. In~\cite{Kantor72} Kantor classified all simple conservative algebras and triple systems of second-order and defined the class of terminal algebras as algebras satisfying some certain identity. He proved that every terminal algebra is a conservative algebra and classified all simple finite-dimensional terminal algebras with left quasi-unit over an algebraically closed field of characteristic zero~\cite{Kantor89term}. Terminal trilinear operations were studied in~\cite{Kantor89tril}, and some questions concerning the classification of simple conservative algebras were considered in~\cite{Kantor89trudy}. After that, Cantarini and Kac classified simple finite-dimensional (and linearly compact) super-commutative and super-anticommutative conservative superalgebras and some generalization of these algebras (also known as ``rigid'' or quasi-conservative superalgebras) over an algebraically closed field of characteristic zero \cite{kac_can_10}.
The classification of all $2$-dimensional conservative and rigid (in sense of Kac-Cantarini) algebras is given in \cite{cfk};
and also, the algebraic and geometric classification of nilpotent low dimensional terminal algebras is given in \cite{kkp19,kks19}.

The algebra $W(n)$ plays a similar role in the theory of conservative algebras as the Lie algebra of all $n\times n$ matrices $\mathfrak{gl}_n$ plays in the theory of Lie algebras. Namely, in~\cite{Kantor88,Kantor90}  Kantor considered the category $\mathcal{S}_n$ whose objects are conservative algebras of non-Jacobi dimension $n$. It was proven that the algebra $W(n)$ is the universal attracting object in this category, i.e., for every $M\in\mathcal{S}_n$ there exists a canonical homomorphism from $M$ into the algebra $W(n)$. In particular, all Jordan algebras of dimension $n$ with unity are contained in the algebra
$W(n)$. The same statement also holds for all noncommutative Jordan algebras of dimension $n$ with unity.
Some properties of the product in the algebra $W(n)$ were studied in \cite{kantorII,kay}.
The universal conservative superalgebra was constructed in \cite{kpp19}.
The study of low dimensional conservative algebras was started in \cite{kaylopo}.
The study of properties of $2$-dimensional algebras is also one of popular topic in non-associative algebras (see, for example, \cite{Giambruno,kayvo2,GR11}) and as we can see the study of properties of the algebra $W(2)$ could give some applications on the theory of $2$-dimensional algebras. So, from the description of idempotents of the algebra $W(2)$ it was received an algebraic classification of all $2$-dimensional algebras with left quasi-unit \cite{kayvo}.
Derivations and subalgebras of codimension 1 of the algebra $W(2)$ and of its principal subalgebras $W_2$ and $S_2$ were described \cite{kaylopo}.
Later, the automorphisms, one-sided ideals, idempotents and local (and $2$-local) derivations and automorphisms of
$W(2)$ and its principal subalgebras were described in \cite{arzikulov,kayvo,consiv}.
Note that $W_2$ and $S_2$ are simple terminal algebras with left quasi-unit from the classification of Kantor \cite{Kantor89term}.

The present paper is devoted to continuing the study of properties of $W(2)$ and its principal subalgebras. {We pay also some attention to the description of the affine group scheme of automorphisms of the algebras under scope, with an eye on the classification of gradings of these algebras (over arbitrary fields), which will deserve a forthcoming paper}.

\section{The graph of an algebra basis}

In this section, the ground field $\F$ will not be assumed to have characteristic zero.
For an arbitrary $\F$-algebra $A$ we will denote by $\M(A)$ (or simply $\M$ if there is no possible ambiguity), the multiplication algebra of $A$, that is, the subalgebra of $\Endo_{\F}(A)$ (where $A$ is considered as a vector space) generated by left and right multiplication operators.
We will denote by $\M_1(A)$ the subalgebra of $\Endo_{\F}(A)$  generated by $1$ and $\M(A)$.
Observe that if $A$ is an algebra whose multiplication algebra is $\M$, and $S\subset A$ a subset, the ideal of $A$ generated by $S$ agrees with $\M S$ (defined as the linear span of the elements $T(x)$ where $T\in\M$ and $x\in S$).

Assume $A$ is an algebra over a field $\F$.  Fix a basis
$(u_i)_{i\in I}$ of $A$. Then we can construct a graph whose vertices are the basic elements $u_i$ and for any two vertices we draw an arrow from $u_i$ to $u_j$ if $u_j=T(u_i)$ for some $T$ in $\M_1(A)$. This relation $T(u_i)=u_j$ will be denoted $u_i\ge u_j$. The relation $\ge$ is reflexive and transitive.
However, to simplify the resulting graph, (1) we will not draw an arrow from each $u_i$ to itself (as we should); and (2) if $u_i\ge u_j\ge u_k$ we will draw an arrow from $u_i$ to $u_j$ and another from $u_j$ to $u_k$ but there will be no need to draw the arrow from $u_i$ to $u_k$. So there are many choices to draw the simplified graph but they all give the same information.

For instance
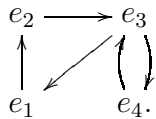
\[
\xygraph{ !{<0cm,0cm>;<1.5cm,0cm>:<0cm,1.2cm>::}
!{(0,0) }*+{e_1}="u"
!{(0,1) }*+{e_2}="a"
!{(1,1) }*+{e_3}="b"
!{(1,0)}*+{e_4.}="c"
"u":"a" "a":"b"
"b":@/^/"c" "b":"u"
"c":@/^/"b"
}
\]
is the graph associated to the algebra $\Sa_2$ whose multiplication table (for a ground field of characteristic other than $3$)
 is
given below

\begin{longtable}{|c|c|c|c|c|c|c|c|c|}
\hline
      & $e_1$ & $e_2$ & $e_3$ & $e_4$ \\ \hline
$e_1$ & $-e_1$ & $-3e_2$ & $e_3$ & $3e_4$ \\ \hline
$e_2$ & $3e_2$ & $0$ & $2e_1$ & $e_3$  \\ \hline
$e_3$ & $-2e_3$ & $-e_1$ & $-3e_4$ & $0$   \\ \hline
$e_4$ & $0$ & $0$ & $0$ & $0$   \\ \hline
\end{longtable}

We can see that the graph is {\em \sc} in the sense that for any two vertices there is a path connecting them. This means that the ideal generated by any $e_i$ is the whole algebra.
In case the ground field has characteristic $3$ the graph is given in figure below,
 \begin{center}
\begin{tikzpicture}
\begin{scope}[every node/.style={thick}]
    \node (A) at (0,0) {$e_3$};
    \node (B) at (0,-1.5) {$e_1$};
    \node (C) at (1.5,-0.8) {$e_4$};
    \node (D) at (-1.5,-0.8) {$e_2$};
\end{scope}
\begin{scope}[>={Stealth[black]}]
    \path [->] (A) edge[bend right=30] node[left] {} (B);
    \path [->] (B) edge[bend right=30] node[left] {} (A);
    \path [->] (C) edge node[left] {} (A);
    \path [->] (D) edge node[left] {} (A);
    \path [->] (D) edge node[left] {} (B);
    \end{scope}
\end{tikzpicture}

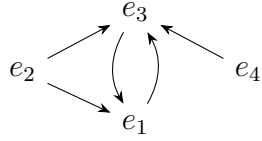
{\it Graph of $\Sa_2$ in case $\ch(\F)=3$.}
\end{center}

\noindent which is not \sc. As we will see, strong connection is a necessary condition for simplicity. If $E$ is any graph with set of vertices $E^0$ and $S\subset E^0$, we will denote by $\tree(S)$ the set
of all $v\in E^0$ such that there is a path from $S$ to $v$.
We can also construct a map $\tree\colon E^0\to E^0$ such that $S\mapsto\tree(S)$. If $E$ is the graph of an $\F$-algebra $B$ relative to a basis ${\mathcal B}=\{b_i\}_i$, the fixed points of $\tree$ induce ideals of $B$: assume $\tree(S)=S$, then $\oplus_{u\in S}\F u$ is a right ideal of $B$ because for any $e_i\in\mathcal{B}$ and any $u\in S\subset\mathcal{B}$, one has $ue_i=0$ or $ue_i=\sum_{j\in J}x^j e_j$ (with $x^j\in\F^\times$) so that $u\ge e_j$ (for any $j\in J$) implying $e_j\in S$. Thus $(\oplus_{u\in S}\F u)B\subset \oplus_{u\in S}\F u$.
Similarly one can prove that $\oplus_{u\in S}\F u$ is a left ideal of $B$. So we claim:
\begin{Lem}
Let $E$ be the graph of an $\F$-algebra $B$ relative to a given basis $\mathcal{B}=(b_i)_i$. If $S$ is a fixed point of $\tree\colon E^0\to E^0$ then $\oplus_{u\in S}\F u$ is an ideal of $B$.
\end{Lem}

When $\ch(\F)=3$, considering the graph of $S_2$ in figure above, we see immediately that the unique fixed points of the map $\tree$ are
the subsets of vertices $\emptyset$, $\{e_1,e_3\}$, $\{e_1,e_2,e_3\}$, $\{e_1,e_3,e_4\}$ and $E^0$. So at a first glace we detect three nontrivial proper ideals: $\F e_1\oplus\F e_3$, $\F e_1\oplus\F e_2\oplus \F e_3$ and $\F e_1\oplus\F e_3\oplus \F e_4$.
If $I$ is the $3$-dimensional ideal generated by $e_1$, $e_3$ and $e_4$, we have
$A=I\oplus\F e_2$ and $A/I$ is a zero-product algebra. Similarly if $J$ is the ideal generated by $\{e_1,e_2,e_3\}$ then $A=J\oplus\F e_4$ and
$A/J$ is a zero product algebra.
With these ideas in mind, an easy criterium for simplicity is

\begin{Lem}\label{ohcum}
Let $A$ be any algebra with $A^2\ne 0$ and let $\M:=\M(A)$ be its multiplication algebra. Then $A$ is simple if and only if
\begin{enumerate}
    \item Its graph relative to a basis $(u_i)_{i\in I}$ is \sc, and
    \item For any nonzero $x\in A$ there is some $u_i$ in $\M x$.
\end{enumerate}
\end{Lem}

For instance to check the simplicity of the four-dimensional algebra $\Sa_2$
whose multiplication table is given  above (in the case of characteristic other than $3$), since its graph is \sc\ we only need to realize that for a nonzero $x=\sum x_i e_i\in \Sa_2$ we have:
\begin{enumerate}
\item If $x_3\ne 0$, $(xe_2)e_2=3x_3 e_2$ hence $e_2\in\M(\Sa_2)x $.
\item If $x_3=0$, $x_1\ne 0$, $xe_2=-3x_1 e_2$, so $e_2\in\M(\Sa_2)x $.
\item If $x_3=x_1=0$ $x_4\ne 0$, $e_2x=x_4e_3$ implying $e_3\in\M(\Sa_2)x$.
\item If $x_i=0$ except for $i=2$ then $e_2\in\M(\Sa_2)x$.
\end{enumerate}
Thus, in any case there is a basis element in $\M(\Sa_2)x$.

For an algebra $A$, the condition that
$\M(A)=\Endo_\F(A)$ implies simplicity of $A$:
indeed, if this coincidence happens, for any nonzero $x\in A$ and any $y\in A$, there is a linear map $f\colon A\to A$ such that $y=f(x)$. Since $f\in\M(A)$ then $y$ is in the ideal generated by $x$. Thus $A$ is simple. In \cite[Corollary of Theorem 3]{jac} it is proved that in the finite-dimensional case, an algebra $U$ over a field $\F$ is simple if and only if its multiplication algebra is simple.

\begin{Lem}\label{caip}
Let $U$ be a finite-dimensional algebra over a field $\F$. If $U$ is simple then $\M(U)={\rm End}_\F(U)$.
\end{Lem}
\begin{proof}
Assume first that the ground field $\F$ is algebraically closed. If $U$ is simple, by \cite{jac} we know that $\M:=\M(U)$ is simple. Let $n:=\dim(U)$, since $U$ is an $\M$-module (irreducible and faithful) then $\M\cong \Endo_\F(U)$. If the ground field $\F$ is not algebraically closed we consider the algebraic closure $\Omega$ of $\F$ and the $\Omega$-algebra $U_\Omega:=U\otimes_\F\Omega$. Then
$\M(U_\Omega)=\Endo_\Omega(U_\Omega)$ and by \cite[(2.5)Lemma]{finston} we have $\M(U_\Omega)\cong\M(U)\otimes\Omega$.
Since $\dim(\Endo_\F(U))=\dim_\Omega(\Endo_\Omega(U_\Omega))=\dim_\Omega(\M(U_\Omega))=\dim(\M(U))$ we conclude
$\Endo_\F(U)=\M(U)$.
\end{proof}

If $U$ is simple but fails to be finite-dimensional we can say that $\M(U)$ is a dense subalgebra of $\Endo_\F(U)$ in the sense of Jacobson density. To clarify this, the $\M$-module $U$ is simple (or irreducible in the terminology of \cite{Jacobson}). Since the action of $\M$ on $U$ is the natural one, we can say that $U$ is an irreducible and faithful $\M$-module. Hence $\M$ is a primitive $\F$-algebra. The irreducibility of $U$ as an $\M$-module implies that the $\F$-algebra
$\Gamma:=\Endo_\M(U)$ is a division algebra.
This consists of all $\F$-linear maps $T\colon U\to U$ such that
$T(xy)=xT(y)=T(x)y$ for any $x,y\in U$. So $\Gamma$ is the centroid of $U$ which is known to be a field (extesion of $\F$) given the simplicity of $U$. Now $U$ is an $\Gamma$-algebra in a natural way
and we have a monomorphism $\M\hookrightarrow\Endo_\Gamma(U)$ which is dense in the sense that for any $\Gamma$-linearly independent $x_1,\ldots, x_n\in U$ and arbitrary $y_1,\ldots, y_n\in U$, there is some $T\in\M$ such that $T(x_i)=y_i$ for $i=1,\ldots n$ (see \cite[Density Theorem for Irreducible Modules, II, \S 2, p.28]{Jacobson}). Observe that when $U$ is finite-dimensional the extension field $\Gamma$ has $(\Gamma:\F)$ finite. Thus if $\F$ is algebraically closed we have $\Gamma=\F$ and $\M$ being dense in $\Endo_\Gamma(U)=\Endo_\F(U)$ gives $\M=\Endo_\F(U)$. If $\F$ is not algebraically closed we can argue as in the last part of the proof
of Lemma \ref{caip}. So we recover Lemma \ref{caip} from the general
result:

\begin{proposition}
If $U$ is a simple $\F$-algebra then $\M:=\M(U)$ is a primitive algebra, more precisely there is:
\begin{enumerate}
\item A monomorphism $\M\hookrightarrow {\rm End}_\Gamma(U)$ where $\Gamma$ is the centroid of $U$ $($which is a field extension of $\F$$)$.
\item For any collection of $\Gamma$-linearly independent elements
$x_1,\ldots,x_n\in U$ and any collection $y_1,\ldots,y_n\in U$, there is an element $T\in\M$ such that $T(x_i)=y_i$ for any $i$.
\end{enumerate}
\end{proposition}

As a consequence of Lemma \ref{caip}, for a finite-dimensional algebra $A$ over a field $\F$, proving that $\Endo_\F(A)$ agrees with $\M(A)$ is equivalent to proving that $A$ is simple.
The characterization of the coincidence $\M(A)=\Endo_\F(A)$ in terms of the graph of $A$ is:

\begin{proposition}\label{atunrojo}
Let $A^2\ne 0$ be a finite-dimensional algebra and $\M=\M(A)$ its multiplication algebra. Then
$\M={\rm End}_{\F}(A)$ if and only if: \begin{enumerate}
    \item\label{atun1} The graph of $A$ relative to a basis $(u_i)_{i=1}^n$ is \sc.
    \item\label{atun2} For every $i\in \{1,\ldots,n\}$ there is some $j\in\{1,\ldots,n\}$ and $T\in\M$ such that $T(u_k)=\delta_{ik}u_j$ for any $k$.
\end{enumerate}
\end{proposition}
\begin{proof} If $\M=\Endo_\F(A)$ the algebra $A$ is simple whence the graph is \sc. The other assertion in the statement is straightforward. So assume that both conditions in the statement hold. If we define the linear maps
$E_{ij}\colon A\to A$ such that $E_{ij}(u_k)=\delta_{ik}u_j$ we know that
$\Endo_\F(A)=\oplus_{i,j=1}^n\F E_{ij}$. Now, condition 2) says that for any $i$ there is some $E_{ij}\in\M$. But the graph relative to $(u_i)$ is \sc\ so
for any $u_j,u_k$ there exists $T\in M$ such that $u_k=T(u_j)$.
Thus $TE_{ij}=E_{ik}$ and we have $E_{ik}\in\M$ for every $k$ and $i$. \end{proof}

\begin{Remark}\label{locomia}\rm
If the graph of an $\F$-algebra $A$ relative to a basis $(u_i)_{i=1}^n$ is strongly connected and some $E_{ij}\in\M=\M(A)$ (identifying $\M$ with an subalgebra of $\Endo_\F(A)$), then $E_{ik}\in\M$ for any $k$. Indeed: given $u_j$ and $u_k$ by the strong connectedness of the graph, there is some $T\in\M$ such that $T(u_j)=u_k$. Then $E_{ik}=TE_{ij}\in\M$.
\end{Remark}

\begin{Th}\label{thm_S2}
If $\F$ is a field of characteristic other than $3$ and $\Sa_2$ the four-dimensional algebra whose multiplication table is given above, we have $\M:=\M(\Sa_2)={\rm End}_\F(\Sa_2)$. Consequently $\Sa_2$ is simple. If the characteristic of $\F$ is $3$ there is a $3$-dimensional ideal $I$ which is the subspace generated by $e_1$, $e_3$ and $e_4$. Moreover
$\Sa_2=I\oplus\F e_2$ and $\Sa_2/I$ is a zero-product algebra.
In this case $\M$ has dimension $8$ and a $4$-dimensional radical $\rad(\M)$ such that $\rad(\M)^2=0$ and $\M/\rad(\M)\cong M_2(\F)$.
\end{Th}
\begin{proof} Since the graph relative to the basis $(e_i)_{i=1}^4$ is \sc\ we need to check \eqref{atun2} in Proposition \ref{atunrojo}.

\begin{enumerate}
    \item[(A)]
First, we will consider
the case in which the characteristic of $\F$ is other than $2$ or $3$.
Under this hypothesis, the element $R_{e_2}^2\in \M$ acts in the way
\begin{longtable}{ll}
$R_{e_2}^2(e_1)=(e_1e_2)e_2=-3e_2^2=0,$ &$ R_{e_2}^2(e_2)=(e_2e_2)e_2=0,$\\
$R_{e_2}^2(e_3)=(e_3e_2)e_2=-e_1e_2=3 e_2,$ & $R_{e_2}^2(e_4)=(e_4e_2)e_2=0.$
\end{longtable}
Thus $E_{32}=\frac{1}{3}R_{e_2}^2\in\M$ and we can also prove that $E_{3k}\in\M$ for any $k$:
\begin{longtable}{lll}
$E_{31}=\frac{1}{2}R_{e_3}E_{32}\in\M,$ &
$E_{33}=R_{e_4}E_{32}\in\M,$ & $E_{34}=-\frac{1}{3}R_{e_3}E_{33}\in\M.$
\end{longtable}
So far $E_{3k}\in\M$ for any $k$. Furthermore, it can be checked that
\begin{longtable}{ll}
$R_{e_1}=-E_{11}+3 E_{22}-2 E_{33},$ &
$L_{e_1}=-E_{11}-3 E_{22}+ E_{33}+3E_{44}.$
\end{longtable}
On the other hand we have:

\begin{longtable}{rcll}
$E_{21}$ & $=$ & $-R_{e_2}R_{e_4}\in\M,$ &\\
$R_{e_2}$ & $=$ & $-3 E_{12}-E_{31},$ & $\hbox{ hence } E_{12}\in\M$\\
$E_{12}, E_{21}$&$\in$ & $\M,$ &  $\hbox{ hence } E_{11}, E_{22}\in\M$\\
$L_{e_1}$ & $=$ & $-E_{11}-3 E_{22}+ E_{33}+3E_{44},$ & $\hbox{ hence } E_{44}\in\M.$
\end{longtable}
Summarizing $E_{ii}\in\M$ for $i=1,2,3,4$. Thus
\eqref{atun2} of Proposition \ref{atunrojo} is satisfied.

    \item[(B)]
Second,  we analyze the case in which $\F$ has characteristic $2$. The multiplication table of $\Sa_2$ adopts the form in figure:

\begin{center}
\begin{longtable}{|c|c|c|c|c|c|c|c|c|}
\hline
      & $e_1$ & $e_2$ & $e_3$ & $e_4$ \\ \hline
$e_1$ & $e_1$ & $e_2$ & $e_3$ & $e_4$ \\ \hline
$e_2$ & $e_2$ & $0$ & $0$ & $e_3$  \\ \hline
$e_3$ & $0$ & $e_1$ & $e_4$ & $0$   \\ \hline
$e_4$ & $0$ & $0$ & $0$ & $0$\\ \hline
\end{longtable}

\it{Multiplication table of $\Sa_2$ when $\hbox{char}(\F)=2.$}

 \end{center}

Then the matrix whose $(i,j)$ entry is $R_{e_i}R_{e_j}$ is shown in
 \begin{equation*}
(R_{e_i}R_{e_j})_{i,j=1}^4=
\left(
\begin{array}{cccc}
 E_{11}+E_{22} & E_{12}+E_{31} & 0 & 0 \\
 E_{12} & E_{32} & E_{11} & E_{21} \\
 E_{13} & E_{33} & E_{14} & E_{24} \\
 E_{14}+E_{23} & E_{13}+E_{34} & 0 & 0 \\
\end{array}
\right),
\end{equation*}
hence $E_{1i}\in\M$ for any $i$. From this, $E_{2i}\in\M$ also for any $i$. Consequently $E_{3i}\in\M$ for any $i$.
The matrix whose $(i,j)$ entry is $L_{e_i}L_{e_j}$ is
\begin{equation*}(L_{e_i}L_{e_j})_{i,j=1}^4= \left(
\begin{array}{cccc}
 E_{11}+E_{22}+E_{33}+E_{44} &
 E_{12}+E_{43} & E_{21}+E_{34} & 0 \\
 E_{12}+E_{43} & 0 & E_{22}+E_{33} & 0
   \\
 E_{21}+E_{34} & E_{11}+E_{44} & 0 & 0
   \\
 0 & 0 & 0 & 0 \\
\end{array}
\right),\end{equation*}
which implies $E_{43},E_{44}\in\M$.
$E_{41},E_{42}\in\M$ follows from the multiplication table.

    \item[(C)]
Third, assume that $\text{char}(\F)=3$. Denoting as before by
$E_{ij}$ the basis of $\Endo_\F(A)$ such that $E_{ij}(e_k)=\delta_{ik}e_j$, we have
\begin{longtable}{lll}
$L_{e_1}=R_{e_1}=-E_{11}+E_{33},$ &
$L_{e_2}=-E_{31}+E_{43},$ &
$R_{e_2}=-E_{31},$\\
$L_{e_3}=R_{e_3}=E_{13}-E_{21},$ &
$R_{e_4}=E_{23}.$
\end{longtable}
The subalgebra of ${\rm End}_\F(A)$ generated by these operators is  $\M$ and coincides with the $\F$-linear span of $\{E_{11}, E_{33}, E_{31},  E_{13}, E_{21}, E_{23}, E_{41}, E_{43}\}$. If we compute the radical of the symmetric bilinear form $\esc{\cdot,\cdot }\colon \M\times\M\to\F$ given by $\esc{f,g}:=\hbox{trace}(f g)$
we find that $$\M^\bot:=\rad(\esc{\cdot,\cdot})=
\F E_{43}\oplus \F E_{21}\oplus \F E_{23}\oplus \F E_{41}.$$
For the reader's convenience, we recall that the radical of a bilinear symmetric form  in a vector space is the subspace of elements which are orthogonal to the whole space. For a finite-dimensional associative algebra, of endomorphisms, the bilinear form $\esc{f,g}:=\hbox{trace}(fg)$ is associative in the sense that $\esc{fg,h}=\esc{f,gh}$. Then, it is easy to realize that $\M^\bot$ is an ideal of the algebra. Now, one can see that \begin{center}$(\M^\bot)^2=0$ and
$\M/\M^\bot\cong \F E_{11}\oplus \F E_{13}\oplus\F E_{31}\oplus \F E_{33}\cong M_2(\F)$.
\end{center} Thus $\M^\bot$ is a maximal ideal and since it is nilpotent we conclude that $\M^\bot$ is the radical of $\M$.

\end{enumerate}\end{proof}

\section{Automorphisms and multiplication algebras}

\subsection{Conservative algebra $S_2$}
In this subsection we compute the automorphism group scheme of $\Sa_2$ over a field $\F$ of arbitrary characteristic.
In order to do that we consider an associative, commutative and unital ring $R$ and define in the free $R$-module $B$ with basis $\{e_i\}_{i=1}^4$ the product as in table of  multiplication of $S_2$ (extended by $R$-bilinearity to the whole $B$). When $R=\F$ the algebra $B$ is precisely $\Sa_2$. So by considering $B$ we are thinking about $A$ defined over an arbitrary ring $R$ (associative, commutative and unital). If we are able to determine $\aut_R(B)$, then we have walked a long way towards the knowledge of the affine group scheme of $A$ over arbitrary fields. So consider $f\in\aut(B)$, and write $f(e_i)=f_i
^j e_j$ (using Einstein Criterium). Then $f(e_1^2)=f(e_1)^2$ hence $$-f(e_1)=(f_1^je_j)^2=-(f_1^1)^2e_1-3(f_1^3)^2e_4-3f_1^1f_1^2e_2+f_1^1f_1^3e_3+3f_1^1f_1^4e_4+$$
$$3f_1^2f_1^1e_2+2f_1^2f_1^3e_1+f_1^2f_1^4 e_3-2f_1^3f_1^1e_3-f_1^3f_1^2e_1=$$
$$[-(f_1^1)^2+f_1^2f_1^3]e_1+[f_1^2f_1^4-f_1^1f_1^3]e_3+[-3(f_1^3)^2+
3f_1^1f_1^4]e_4,$$
so we deduce
$$\begin{cases}
f_1^1=(f_1^1)^2
, & f_1^2=0\cr
f_1^3=f_1^1f_1^3
, & f_1^4=3(f_1^3)^2-3f_1^1f_1^4
\end{cases}$$
 Furthermore since $Re_4$ is the left annihilator of $B$ (and this is preserved under automorphism) we have $f(e_4)=f_4^4 e_4$ which implies $f_4^4\in R^\times$ (invertible elements of $R$). Now applying $f$ to $e_1e_4=3e_4$ we get
$(f_1^1e_1+f_1^3e_3+f_1^4e_4)e_4=3e_4$, that is, $3f_1^1e_4=3e_4$. Assume now that $\tor_3(R)=0$, then $f_1^1=1$ so that $4f_1^4=3(f_1^3)^2$.

Then we discuss cases:
\begin{enumerate}
    \item[(A)] $\frac{1}{3},\frac{1}{2}\in R$. Then $f_1^4=\frac{3}{4}(f_1^3)^2$ (besides $f_1^1=1, f_1^2=0$). Since $e_3e_4=0$ we have $f(e_3)e_4=0$ hence $f_3^1e_1e_4+f_3^2e_2e_4=0$. So $3f_3^1e_4+f_3^2 e_3=0$ implying $f_3^1=f_3^2=0$. So far the matrix of $f$ relative to the $R$-basis $\{e_i\}$ is
    \begin{equation} \label{pattern}
        \tiny \begin{pmatrix}1 & 0 & * & *\cr
    f_2^1 & f_2^2 & f_2^3 & f_2^4\cr
    0 & 0 & f_3^3 & *\cr
    0 & 0 & 0 & f_4^4\end{pmatrix}
    \end{equation}
whose determinant is $f_2^2f_3^3f_4^4$ and must be in $R^\times$. Hence $f_i^i\in R^\times$ for any $i$. Also
$e_3^2=-3e_4$ and applying $f$ we have
$(f_3^3e_3+f_3^4e_4)^2=-3 f_4^4 e_4$. Thus $-3(f_3^3)^2e_4=-3 f_4^4 e_4$ which gives $f_4^4=(f_3^3)^2$. Take now into account that $e_3e_2=-e_1$ hence
\begin{center}
$(f_3^3 e_3+f_3^4e_4)(f_2^ie_i)=-(e_1+f_1^3 e_3+f_1^4 e_4)$
\end{center}or equivalently
$f_3^3 e_3(f_2^ie_i)=-(e_1+f_1^3 e_3+f_1^4 e_4)$. So
$$-2f_3^3f_2^1 e_3-f_3^3f_2^2e_1-3f_3^3f_2^3 e_4=-e_1-f_1^3e_3-f_1^4 e_4$$
and we get $$f_2^2f_3^3=1,\quad f_1^3=2f_3^3f_2^1,\quad f_1^4=3f_3^3f_2^3.$$
Now $e_2^2=0$ hence $(f_2^ie_i)^2=0$. Thus $$-(f_2^1)^2e_1-3(f_2^3)^2e_4
-3f_2^1f_2^2 e_2+f_2^1f_2^3e_3+3f_2^1f_2^4e_4+$$
$$3f_2^2f_2^1e_2+2f_2^2f_2^3e_1+f_2^2f_2^4e_3-2f_2^3f_2^1e_3-f_2^3f_2^2e_1=0.$$ We get
$$-(f_2^1)^2+f_2^2f_2^3=0,\ f_2^2f_2^4-f_2^3f_2^1=0,\ -3(f_2^3)^2+3f_2^1f_2^4=0.$$
Also $e_1e_3=e_3$ so that
$(e_1+f_1^3e_3+f_1^4 e_4)f_3^i e_i= f_3^i e_i$. Equivalently
\begin{center}$(e_1+f_1^3e_3)(f_3^3 e_3+f_3^4 e_4)= f_3^3 e_3+f_3^4 e_4$.\end{center}
Then
$f_3^3 e_3+3f_3^4 e_4-3f_1^3f_3^3e_4=f_3^3 e_3+f_3^4 e_4$ and we get $2f_3^4 -3f_1^3f_3^3=0$ so that $f_3^4=\frac{3}{2}f_1^3f_3^3$.
Thus if we put $f_2^2=\lambda$ and $f_1^3=\mu$ we have $f_3^3=\frac{1}{\lambda}$ and
\begin{longtable}{lll}
$f_1^4=\frac{3}{4}\mu^2,$ &
$f_2^1=\frac{\lambda\mu}{2},$ &
$ f_2^3=\frac{\lambda\mu^2}{4},$\\
$f_2^4= \frac{\lambda\mu^3}{8},$ &
$f_3^1=f_3^2=0,$ &
$f_3^4=\frac{3\mu}{2\lambda}.$
\end{longtable}
Thus the matrix of $f$ in
the basis $\{e_i\}$ is
\begin{equation}\label{tsrifeht}
\tiny   w(\lambda,\mu)=\left(
\begin{array}{cccc}
 1 & 0 & \mu  & \frac{3 \mu
   ^2}{4} \\
 \frac{\lambda  \mu }{2} &
   \lambda  & \frac{\lambda
   \mu ^2}{4} & \frac{\lambda
   \mu ^3}{8} \\
 0 & 0 & \frac{1}{\lambda } &
   \frac{3 \mu }{2 \lambda }
   \\
 0 & 0 & 0 & \frac{1}{\lambda
   ^2} \\
\end{array}
\right),\quad \lambda\in R^\times, \mu\in R,
\end{equation}
and we can check that any $f$ whose matrix is $w(\lambda,\mu)$ is an automorphism of $B$. Also the formula $w(\lambda,\mu)w(\lambda',\mu')=w(\lambda\lambda',\mu'+\mu/\lambda')$ gives that $$\aut(B)=\{w(\lambda,\mu)\colon\lambda\in R^\times,\mu\in R\}\cong \begin{pmatrix}1 & R\cr 0 & R^\times\end{pmatrix}$$ being the isomorphism $w(\lambda,\mu)\mapsto\tiny\begin{pmatrix}1 & \mu \cr 0 & \lambda^{-1}\end{pmatrix}$. So $\aut(B)$ is isomorphic to the
group $\Aff_2(R)$ of all invertible affine transformations of the affine plane
${\mathbb A}_2(R)$. In \cite{kayvo} it is proved this result in the particular case that $R$ is a field $\F$ of characteristic zero.

\item[(B)] $\frac{1}{3}\in R$ but $2R=0$. We have $f_1^1=1$, $f_1^2=0$ and $(f_1^3)^2=0$. Recall also that $f_4
 ^i=0$ for $i\ne 4$. Taking into account the multiplication table,
 which is the same that the multiplication of $S_2$ when $\hbox{char}(\F)=2$, we deduce (from $e_3e_4=0$) that $f_3^1 e_4+f_3^2 e_3=0$ so $f_3^1=f_3^2=0$. Consequently we have the same pattern \eqref{pattern} for the matrix of the automorphism. So $f_2^2$, $f_3^3$, $f_4^4$ are invertible. Then we deduce
$f_4^4=(f_3^3)^2$ as in the previous case. Also following the argument in the previous case we get
$$f_2^2f_3^3=1,\ f_1^3=0,\ f_1^4=f_3^3f_2^3.$$
Now, from $e_2^2=0$ we get:
$$(f_2^1)^2+f_2^2f_2^3=0,\ f_2^2f_2^4-f_2^3f_2^1=0,\ (f_2^3)^2+f_2^1f_2^4=0.$$
If we put $\lambda=f_2^2$, $\mu=f_2^1$ we have
$f_2^3=(f_2^1)^2/f_2^2=\mu^2/\lambda$ implying $f_1^4=\mu^2/\lambda^2$. Then $f_3^3=1/\lambda$. On the other hand, since $e_2e_4=e_3$ we have
$f_2^ie_i f_4^4 e_4=f_3^3e_3+f_3^4 e_4$ which gives
$f_2^1f_4^4 e_4+f_2^2f_4^4e_3=f_3^3e_3+f_3^4 e_4$, that is, $f_2^1f_4^4=f_3^4$ and $f_2^2f_4^4=f_3^3$.
So $f_3^4=\mu/\lambda^2$ and $f_4^4=1/\lambda^2$. The matrix of $f$ is
\begin{equation}\label{dnoceseht}
w_2(\lambda,\mu)=\tiny\left(
\begin{array}{cccc}
 1 & 0 & 0 & \frac{\mu ^2}{\lambda ^2} \\
 \mu  & \lambda  & \frac{\mu ^2}{\lambda } & \frac{\mu ^3}{\lambda
   ^2} \\
 0 & 0 & \frac{1}{\lambda } & \frac{\mu }{\lambda ^2} \\
 0 & 0 & 0 & \frac{1}{\lambda ^2} \\
\end{array}
\right),\quad \lambda\in R^\times, \mu\in R,
\end{equation}
and reciprocally: if the matrix of $f$ relative to the basis $\{e_i\}$ is as above, then $f$ is an automorphism of $B$. We also have the formula $w_2(\lambda, \mu) w(\lambda', \mu')=w(\lambda\lambda', \mu + \lambda\mu')$
which gives an isomorphism
$$\aut(B)=\{w_2(\lambda,\mu)\colon\lambda\in R^\times, \mu\in R\}\cong\begin{pmatrix} 1 & R \cr 0 & R^\times\end{pmatrix}$$
being the isomorphism the given by $w_2(\lambda,\mu)\mapsto\tiny\begin{pmatrix}1 & \mu \cr 0 & \lambda\end{pmatrix}$. So again $\aut(B)$ is isomorphic
 to the   group $\Aff_2(R)$ of all invertible affine transformations of the affine plane ${\mathbb A}_2(R)$.

\item[(C)] $3R=0$. As in previous cases $f(e_4)=f_4^4 e_4$ so $f_4^i=0$ except for $i=4$. Imposing the condition $f(e_1^2)=f(e_1)^2$ we get $f_1^2=f_1^4=0$. Imposing $f(e_1e_3)=f(e_1)f(e_3)$ we get
$f_3^2=f_3^4=0$ and from $f(e_2e_4)=f(e_2)f(e_4)$ we get $f_3^1=0$ and $f_3^3=f_2^2f_4^4$.
Taking into account these values, from $f(e_3e_1)=f(e_3)f(e_1)$ we get $f_1^1=1$. Aplying now $f(e_3e_2)=f(e_3)f(e_2)$ gives  $f_1^3=-f_2^1f_3^3$ and $f_4^4=1/(f_2^2)^2$. Using again $f(e_1e_2)=f(e_1)f(e_2)$ we get
$f_2^3=(f_2^1)^2/f_2^2$.
Finally, equation $f(e_2)^2=0$ gives $f_2^4=f_2^1f_2^3/f_2^2=(f_2^1)^3/(f_2^2)^2$.
If we do $\lambda=f_2^2$, $\mu=f_2^1$ we have the matrix \begin{equation}\label{drihteht}\tiny
w_3(\lambda,\mu)=
\begin{pmatrix}
1 & 0 & -\frac{\mu}{\lambda} & 0\cr
\mu & \lambda & \frac{\mu^2}{\lambda} & \frac{\mu^3}{\lambda^2} \cr
0 & 0 & \frac{1}{\lambda} & 0\cr
0 & 0 & 0 & \frac{1}{\lambda^2}
\end{pmatrix}
\end{equation}
so that the equality
$w_3(\lambda,\mu)w_3(\lambda',\mu')=w_3(\lambda\lambda',\mu+\lambda\mu')$ holds. Then $\aut(B)=\{w_3(\lambda,\mu)\colon\lambda\in R^\times,\mu\in R\}\cong\tiny \begin{pmatrix} 1 & R\cr 0 & R^\times\end{pmatrix}$, the isomorphism being
$w_3(\lambda,\mu)\mapsto \tiny\left(
\begin{array}{cc}
 1 & \mu \\
 0  & \lambda  \\
\end{array}
\right).$
 So again $\aut(B)$ is isomorphic
 to the   group $\Aff_2(R)$ of all invertible affine transformations of the affine plane ${\mathbb A}_2(R)$.\end{enumerate}
Thus we claim:
\begin{proposition}\label{reima}
Let $R$ be a commutative associative unital ring and $B$ the free $R$-module with basis $\{e_i\}_{i=1}^4$ endowed with an $R$-algebra structure whose multiplication algebra is
that of table of multiplication of $S_2.$ Then
if $\frac{1}{2},\frac{1}{3}\in R$; or $\frac{1}{3}\in R$, $2R=0$; or $3R=0$, we have $\aut(B)\cong\Aff_2(R)$ the affine group of $\mathbb{A}_2(R)$. The precise description of $\aut(B)$ is given in formulae \eqref{tsrifeht},\eqref{dnoceseht} and \eqref{drihteht}. \end{proposition}

Now fix an arbitrary field $\F$ and let $\Sa_2$ be the $\F$-algebra introduced in the table of multiplication of $S_2$. We can describe the affine group scheme $\affaut(\Sa_2)$.
Denote by $\alg_{\Fs}$  the category of associative, commutative and unital $\F$-algebras and by $\grp$ the category of groups.
Then $\affaut(\Sa_2)$ is the group functor
$\affaut(\Sa_2)\colon\alg_\Fs\to\grp$ such that $R\mapsto \aut((\Sa_2)_R)$ (where $(\Sa_2)_R:=\Sa_2\otimes_\Fs R$ is the scalar extension algebra).
If $\ch(\F)\ne 2,3$ then $\frac{1}{2},\frac{1}{3}\in R$ for any $R\in\alg_\Fs$. If $\ch(\F)=2$ then $2R=0$ for any $R\in\alg_\Fs$ but $\frac{1}{3}\in R$. Finally, if $\ch(\F)=3$ then $3R=0$. So in any case we can apply Proposition \ref{reima} to $R\in\alg_\Fs$ to compute the affine group scheme $\affaut(\Sa_2)$.
Denote by $\hbox{\bf Aff}_2\colon\alg_\Fs\to\grp$ the group functor
such that $R\mapsto \Aff_2(R)=\tiny\begin{pmatrix}1 & R\cr 0 & R^\times\end{pmatrix}$. Then we claim
\begin{Th}\label{analuf}
For an arbitrary field $\F$, there is an isomorphism of group schemes $$\affaut(\Sa_2)\cong\hbox{\bf Aff}_2.$$
\end{Th}
We recall that the isomorphism condition between group functors is that there is a collection of group isomorphisms $\tau_R\colon\aut((\Sa_2)_R)\cong \Aff_2(R)$ such that when $\alpha\colon R\to S$ is an $\F$-algebra homomorphism, the following squares commute:
\begin{center}\small
\begin{tikzcd}[column sep=small]
\aut((\Sa_2)_R) \arrow[r, "\tau_R"] \arrow[d,"\alpha_1"']
& \Aff_2(R) \arrow[d, "\alpha_2" ] \\
\aut((\Sa_2)_S) \arrow[r, "\tau_S"' ]
& \Aff_2(S)
\end{tikzcd}
\end{center}
where each $\alpha_i$ ($i=1,2$) is given by applying the corresponding functor $\affaut(\Sa_2)$ or $\hbox{\bf Aff}_2$ to the homomorphism $\alpha$.
Consider now the $\F$-algebra of dual numbers $\F(\epsilon)$.
Recall that for an algebraic group $\mathcal{G}\subset\GL(V)$ (with $V$ a finite-dimensional $\F$-vector space),  its Lie algebra is
$\lie(\mathcal{G})=\{d\in\gl(V)\colon 1+\epsilon d\in\mathfrak{G}(\F(\epsilon)) \}$.
Thus $\lie(\affaut{\Sa_2})\cong\lie(\hbox{\bf Aff}_2)\cong\aff_2(\F)$
where
$$\aff\nolimits_2(\F)=\left\{\begin{pmatrix}0 & \mu\cr 0 & \l\end{pmatrix}\colon \l,\mu\in\F\right\}.$$
As a corollary of Proposition \ref{analuf} we have
\begin{corollary}
For an arbitrary field we have:
$$\mathfrak{Der}(\Sa_2)\cong\aff\nolimits_2(\F).$$
\end{corollary}

\subsection{Conservative algebra $W_2$}\label{zoro}

Consider now the six-dimensional $\F$-algebra $W_2$ whose multiplication algebra is given in the following table
\begin{center}
\begin{longtable}{|c|c|c|c|c|c|c|}
\hline
      & $e_1$ & $e_2$ & $e_3$ & $e_4$ & $e_5$ & $e_6$ \\ \hline
$e_1$ & $-e_1$ & $-3e_2$ & $e_3$ & $3e_4$ & $-e_5$ & $e_6$  \\ \hline
$e_2$ & $3e_2$ & $0$ & $2e_1$ & $e_3$ & $0$ & $-e_5$ \\ \hline
$e_3$ & $-2e_3$ & $-e_1$ & $-3e_4$ & $0$ & $e_6$ & $0$  \\ \hline
$e_4$ & $0$ & $0$ & $0$ & $0$ & $0$ & $0$ \\ \hline
$e_5$ & $-2e_1$ & $-3e_2$ & $-e_3$ & $0$ & $-2e_5$ & $-e_6$ \\ \hline
$e_6$ & $2e_3$ & $e_1$ & $3e_4$ & $0$ & $-e_6$ & $0$ \\ \hline
\end{longtable}
\end{center}
If $\F$ is of characteristic $\ne 2,3$, the graph of $B$ relative to the basis of the $e_i$'s is
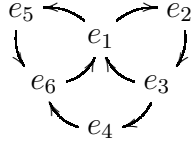
\[
\xygraph{ !{<0cm,0cm>;<1.5cm,0cm>:<0cm,1.2cm>::}
!{(0,0) }*+{e_1}="a"
!{(.7,0.3) }*+{e_2}="b"
!{(.5,-.5) }*+{e_3}="c"
!{(-.7,.3) }*+{e_5}="d"
!{(-.5,-.5) }*+{e_6}="e"
!{(0,-1)}*+{e_4}="f"
"a":@/^/"b" "b":@/^/"c" "c":@/^/"a"
"a":@/_/"d" "d":@/_/"e" "e":@/_/"a"
"c":@/^/"f" "f":@/^/ "e"
}
\]
which is of course strongly connected. In the case $\hbox{char}(\F)=3$ the graph
is not strongly connected:
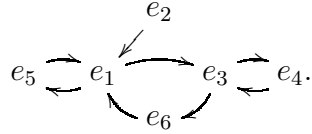
\[
\xygraph{ !{<0cm,0cm>;<1.5cm,0cm>:<0cm,1.2cm>::}
!{(0,0) }*+{e_1}="a"
!{(1,0) }*+{e_3}="b"
!{(.5,-.5) }*+{e_6}="c"
!{(-.7,0) }*+{e_5}="d"
!{(1.7,0) }*+{e_4.}="e"
!{(0.5,.7)}*+{e_2}="f"
"a":@/^/"b" "b":@/^/"c" "c":@/^/"a"
"d":@/^/"a"
"a":@/^/"d"
"e":@/^/"b"
"f":"a"
}
\]
There is an ideal $I:=\oplus_{i\ne 2}\F e_i$ so that $W_2=I\oplus\F e_2$ and again
$W_2/I\cong\F e_2$ is a zero-product algebra.
\newline
\begin{Remark}\label{fooo}\rm
Recall that the Jacobson radical $\rad(A)$ of a unital algebra $A$ contains
every nilpotent ideal of $A$. On the other hand if $I\triangleleft A$ and
$A/I$ is semisimple, then $\rad(A)\subset I$.
\end{Remark}

\begin{Th}\label{siri}
For a ground field $\F$ of characteristic not $3$, we have $\M(W_2)={\rm End}_\F(W_2)$, hence $W_2$ is simple. In the characteristic $3$ case, $W_2$ has a five dimensional ideal $I=\oplus_{i\ne 2}\F e_i$ and $W_2/I\cong\F e_2$ has zero product. The multiplication algebra $\M=\M(W_2)$ is $20$-dimensional, its radical $\rad(\M)$ is $12$-dimensional
and $\M/\rad(\M)=M_2(\F)\oplus M_2(\F)$.
\end{Th}
\begin{proof}
We will apply Proposition \ref{atunrojo} repeatedly.
\begin{enumerate}
    \item[(A)]
Assume first that $\hbox{char}(\F)\ne 2,3$. Then
$L_{e_2}^3=6 E_{42}$ hence $E_{4k}\in\M:=\M(W_2)$ for any $k$ (take into account Remark \ref{locomia}).
Since $L_{e_3}^3=-6E_{24}$ we have $E_{2k}\in\M$ for any $k$.
Also $L_{e_2}^2L_{e_5}^3=-6E_{32}$ hence $E_{3k}\in\M$ for every $k$.
Since $R_{e_2}^2=3 E_{32}-3 E_{62}$ we conclude $E_{6k}\in\M$ for every $k$.
Furthermore $L_{e_2}=3E_{12}+2E_{31}+E_{43}-E_{65}$ hence $E_{1k}\in\M$ for any $k$. And finally  $R_{e_2}=-3E_{12}-E_{31}-3E_{52}+E_{61}$ which implies
$E_{5k}\in\M$ for every $k$. Thus $\M=\Endo_\F(W_2)$.

   \item[(B)] When the ground field has characteristic $2$ we take into account:
\begin{longtable}{rclll}
$R_{e_2}^2L_{e_2}$ &$ =$&  $E_{42}$ & \hbox{ implying } & $E_{4k}\in\M \hbox{ for any } k.$\\
$R_{e_2}L_{e_2}$ &$ =$ & $E_{41}+E_{62}$ & \hbox{ implying } & $E_{6k}\in\M \hbox{ for any } k.$\\
$L_{e_2}$ &  $=$ & $E_{12}+E_{43}+E_{65}$ & \hbox{ implying } & $E_{1k}\in\M \hbox{ for any } k.$\\
$R_{e_4}$ & $=$ & $E_{14}+E_{23}$ & \hbox{ implying } & $E_{2k}\in\M
\hbox{ for any } k.$\\
$L_{e_5}$ & $=$ & $E_{22}+E_{33}+E_{66}$ & \hbox{ implying }  & $E_{3k}\in\M \hbox{ for any } k.$\\
$L_{e_1}$ & $ = $& $\sum_1^6 E_{ii}$ & \hbox{ implying } & $ E_{5k}\in\M \hbox{ for any } k.$
\end{longtable}

    \item[(C)]
In case $\text{char}(\F)=3$ we have
\begin{longtable}{ll}
$L_{e_1}= -E_{11}+E_{33}-E_{55}+E_{66}$, &
$R_{e_1}= -E_{11}+E_{33}+E_{51}-E_{63}$,\\
$L_{e_2}= -E_{31}+E_{43}-E_{65}$, &
$R_{e_2}= E_{61}-E_{31},$\\
$L_{e_3}= E_{13}-E_{21}+E_{56}$, &
$R_{e_3}= E_{13}-E_{21}-E_{53}$,\\
$L_{e_4}= 0$, & $R_{e_4}= E_{23}$,\\
$L_{e_5}= E_{11}-E_{33}+E_{55}-E_{66}$, &
$R_{e_5}= -E_{15}+E_{36}+E_{55}-E_{66}$,\\
$L_{e_6}= -E_{13}+E_{21}-E_{56}$, &
$R_{e_6}= E_{16}-E_{25}-E_{56}$.
\end{longtable}

A basis for $\M$ is given by the set of matrices:

\begin{longtable}{llll}
$E_{11}+E_{55},$ &
$E_{33}+E_{66},$ &
$E_{31}+E_{65},$ &
$E_{13}+E_{56},$ \\

$E_{11}-E_{51},$ &
$E_{33}-E_{63},$ &
$E_{31}-E_{61},$ &
$E_{13}-E_{53},$ \\

$E_{15}-E_{55},$ &
$E_{36}-E_{66},$ &
$E_{16}-E_{56},$ &
$E_{35}-E_{65},$ \\

$E_{21},$ &
$E_{23},$ &
$E_{25},$ &
$E_{26},$ \\

$E_{41},$ &
$E_{43},$ &
$E_{45},$ &
$E_{46}.$

\end{longtable}

We have computed again the radical  $\M^\bot$ of its trace form
$\esc{f,g}:=\tr(fg)$ and it is $12$-dimensional ideal.
More precisely
\begin{longtable}{lcl}
$\M^\bot:=\rad(\esc{\cdot,\cdot})$ &$=$ & $
\F E_{21}\oplus \F E_{23}\oplus \F E_{25}\oplus \F E_{26}\oplus\F  E_{41}\oplus\F E_{43}\oplus\F E_{45}\oplus\F E_{46}\oplus$\\
&& $\F(E_{11}+E_{15}-E_{51}+2 E_{55})\oplus
\F(E_{13}+E_{16}-E_{53}+2 E_{56})\oplus$\\
&&$\F(E_{31}+E_{35}-E_{61}+2 E_{65})\oplus
\F(E_{33}+E_{36}-E_{63}+2 E_{66})$
\end{longtable}

and $(\M^\bot)^2=\F(E_{43}+E_{46})\oplus \F(E_{41}+E_{45})+\F(E_{23}+E_{26})+\F(E_{21}+E_{25})$ being $(\M^\bot)^4=0$. Since $\M^\bot$ is nilpotent, $\rad(\M)\supset\M^\bot$ (see Remark \ref{fooo}).
Define next the subspace $S$ of $\M$ whose basis is
$\{e_{ij}\}_{i,j=1}^2\sqcup \{u_{ij}\}_{i,j=1}^2$ given by
\begin{longtable}{llll}
$e_{1,1}=E_{11}+E_{15},$ &
$e_{2,2}=E_{33}+E_{36},$ &
$e_{1,2}=E_{13}+E_{16},$ &
$e_{2,1}=E_{31}+E_{35},$\\
$u_{1,1}=-E_{15}+E_{55},$ &
$u_{2,2}=-E_{36}+E_{66},$ &
$u_{1,2}=E_{16}-E_{56},$ &
$u_{2,1}=E_{35}-E_{65}.$
\end{longtable}
If $\delta_{ij}$ denotes the Kronecker delta, it is easy to check that $e_{ij}e_{kl}=\delta_{jk}e_{il}$,
$u_{ij}u_{kl}=\delta_{jk}u_{il}$ and $e_{ij}u_{kl}=u_{kl}e_{ij}=0$ for any $i,j,k,l\in\{1,2\}$.
Thus $S\cong M_2(\F)\oplus M_2(\F)$ and furthermore $\M=\M^\bot\oplus S$. Thus $\M/\M^\bot\cong M_2(\F)\oplus M_2(\F)$ is semisimple which implies $\rad(\M)\subset \M^\bot$.
So $\rad(\M)=\M^\bot$.

\end{enumerate}
\end{proof}

\subsubsection{Automorphisms of $W_2$}\label{guaydel}
In this section we work over a commutative ring $R$ and denote $W_2(R)$ the $R$-algebra $\oplus_{i=1}^6 R e_i$ where the multiplication table of the $e_i$'s is that of the multiplication table of $W_2$ given above.
 If we take a generic element $w=\sum_{i=1}^6\l_i e_i\in W_2(R)$ and compute the matrix of $L_w$
relative to the basis of the $e_i$'s we obtain:
{\tiny$$\left(
\begin{array}{cccccc}
 -\lambda _1-2 \lambda _5 & 3 \lambda _2 & 2 \lambda _6-2 \lambda _3 & 0 & 0 & 0 \\
 \lambda _6-\lambda _3 & -3 \lambda _1-3 \lambda _5 & 0 & 0 & 0 & 0 \\
 2 \lambda _2 & 0 & \lambda _1-\lambda _5 & 3 \lambda _6-3 \lambda _3 & 0 & 0 \\
 0 & 0 & \lambda _2 & 3 \lambda _1 & 0 & 0 \\
 0 & 0 & 0 & 0 & -\lambda _1-2 \lambda _5 & \lambda _3-\lambda _6 \\
 0 & 0 & 0 & 0 & -\lambda _2 & \lambda _1-\lambda _5 \\
\end{array}
\right).$$}
If $f$ is any element in $\aut(W_2(R))$ we know
$L_{f(w)}=f L_w f^{-1}$ hence the characteristic polynomial of $L_w$ is invariant under automorphism of $W_2(R)$. So the coefficients of that polynomial are invariants and we list here:

\begingroup\footnotesize\begin{longtable}{lcl}
$\ll_1(w)$&$:=$&$\ 9\l_5,$\\

$\ll_2(w)$&$:=$&$ -11 \lambda _1^2-11 \lambda _5 \lambda _1+31 \lambda _5^2+11 \lambda _2 \lambda _3-11 \lambda _2 \lambda _6,$ \\

$\ll_3(w)$&$:=$&$ -3 \lambda _5 \left(22 \lambda _1^2+22 \lambda _5 \lambda _1-17 \lambda _5^2-22 \lambda _2 \lambda _3+22 \lambda _2 \lambda _6\right),$\\

$\ll_4(w)$&$:=$&$19 \lambda _1^4+38 \lambda _5 \lambda _1^3-120 \lambda _5^2 \lambda _1^2-38 \lambda _2 \lambda _3 \lambda _1^2+38 \lambda _2 \lambda _6 \lambda _1^2-139 \lambda _5^3 \lambda _1-38 \lambda _2 \lambda _3 \lambda _5 \lambda _1+$\\
\multicolumn{3}{r}{ $38 \lambda _2 \lambda _5 \lambda _6 \lambda _1+40 \lambda _5^4+19 \lambda _2^2 \lambda _3^2+139 \lambda _2 \lambda _3 \lambda _5^2+19 \lambda _2^2 \lambda _6^2-139 \lambda _2 \lambda _5^2 \lambda _6-38 \lambda _2^2 \lambda _3 \lambda _6,$}\\

$    \ll_5(w)$&$:=$&$3 \lambda _5 \left(\lambda _1^2+\lambda _5 \lambda _1-2 \lambda _5^2-\lambda _2 \lambda _3+\lambda _2 \lambda _6\right) \left(19 \lambda _1^2+19 \lambda _5 \lambda _1-2 \lambda _5^2-19 \lambda _2 \lambda _3+19 \lambda _2 \lambda _6\right),$\\

$\ll_6(w)$&$:=$&$-9 \left(\lambda _1^2+\lambda _5 \lambda _1-\lambda _2 \lambda _3+\lambda _2 \lambda _6\right) \left(\lambda _1^2+\lambda _5 \lambda _1-2 \lambda _5^2-\lambda _2 \lambda _3+\lambda _2 \lambda _6\right){}^2.$
\end{longtable}\endgroup{}
All these polynomial remain invariant under automorphism but we are using only the first one: $\ll_1$.

\begin{Lem}\label{ecniuq}
If an element $x\in W_2(R)\setminus\{0\}$ with
$\ll_1(x)=9$ satisfies $x^2=-2x$ and $\tor_3(R)=0$ then $x=s e_4+e_5+r e_6$ for some $r,s\in R$ such that $2s=0$. In particular if $\frac{1}{2}\in R$ we have $x=e_5+r e_6$.
\end{Lem}
\begin{proof}
Take $x=x_i e_i$ where $x_i\in R$. Since $\ll_1(x)=9$ we have $x_5=1$. By using the multiplication table, from the equality $x^2=-2x$ we also obtain
\begin{longtable}{lll}
$-x_1^2+x_2 x_3+x_2 x_6=0,$ & $x_2=0,$ & $-x_1 x_3+x_3+x_2 x_4+2 x_1 x_6=0,$\\
$-3 x_3^2+3 x_6 x_3+3 x_1 x_4+2 x_4=0,$ & $-x_1-x_2 x_6=0,$ & $x_3+x_1 x_6=0,$
\end{longtable}
which can be summarized in $x_1=x_2=x_3=2x_4=0$. Thus taking $s=x_4$ and $r=x_6$ the Lemma is proved.
\end{proof}

If $\theta\in\aut(W_2(R))$ then $\theta(e_5)^2=-2\theta(e_5)$ and $\ll_1(\theta(e_5))=\ll_1(e_5)=9$. Hence Lemma \ref{ecniuq} implies that if $\tor_3(R)=0$ then
$\theta(e_5)=s e_4+e_5+r e_6$ with $2s=0$. In case $\tor_2(R)=0$ we have
$\theta(e_5)=e_5+r e_6$.
\begin{Lem}\label{elady}
Let again $\tor_3(R)=0$.
If $x,y\in W_2(R)$ are linearly independent with
$\ll_1(x)=9$, $\ll_1(y)=0$, and they satisfies $x^2=-2x$, $xy=yx=-y$, $y^2=0$ then in case $\frac{1}{2}\in R$ we have $x=e_5+r e_6$ and $y=t e_6$ for some $r,t\in R$.
If $\tor_2(R)=R$ we can only conclude that $x=s e_4+e_5+r e_6$ and $y=t e_6$ for some
$t\in R$.
\end{Lem}
\begin{proof}
From Lemma \ref{ecniuq} we know that if $\frac{1}{2}\in R$ we have $x=e_5+r e_6$. Now writing
$y=\sum \m_i e_i$ (with $\m_5=0$) and imposing $xy=yx=-y$ we get the equations:

\begin{longtable}{lllll}
$  2 r \mu _1=0,$ & $-3 \mu _2=0,$ & $-2 \mu _2=0,$ & $r \mu _2-2 \mu _1=0,$ & $r \mu _2-\mu _1=0,$
\\
$\mu _1+r \mu _2=0,$ & $-r \mu _1-\mu _3=0,$ & $2 r \mu _1-\mu _3=0,$ & $3 r \mu _3=0,$ & $3 r \mu _3+\mu _4=0.$
\end{longtable}

If $\frac{1}{2}\in R$ or $\tor_2(R)=R$, the first row of equations above implies $\m_2=0$, $\m_1=0$. Now the second row gives $\m_3=\m_4=0$. Therefore $y=\m_6 e_6$. If $\tor_2(R)=R$ imposing the conditions we only get $x=s e_4+e_5+r e_6$ and $y=t e_6$.
\end{proof}

For any $\theta\in\aut(W_2(R))$ we can apply Lemma \ref{elady} taking
$x=\theta(e_5)$ and $y=\theta(e_6)$. Thus, if $\frac{1}{2}\in R$ or
$\tor_2(R)=R$ we have
$\theta(e_5)=e_5+re_6$ and $\theta(e_6)=t e_6$ with $r,t\in R$ and $t$ invertible.
In case $\tor_2(R)=R$ we can argue as follows: $\theta(e_2)\theta(e_6)=\theta(e_5)$ so $t\theta(e_2)e_6=s e_4+e_5+r e_6$ with $t$ invertible. But in the image of $R_{e_6}$ is the $R$-submodule $Re_4\oplus R_6$ hence $s=0$. Consequently
\begin{Lem}
Assume that $\tor_3(R)=0$. Then if $\frac{1}{2}\in R$ or $\tor_2(R)=R$
we have $\theta(e_5)=e_5+re_6$ and $\theta(e_6)=t e_6$ with $r,t\in R$ and $t$ invertible.
\end{Lem}

It can be checked that the left annihilator of $W_2(R)$ is the $R$-submodule of all elements $a_1(e_1+e_5)+a_3 (e_3+e_6)+a_4 e_4$ such that $a_i\in R$ with $3a_1=0$:
\begin{equation}\label{rodaluna}
\hbox{\rm Lann}(W_2(R))=\tor\nolimits_3(R)(e_1+e_5)\oplus R(e_3+e_6)\oplus R e_4 \end{equation}

\begin{Lem}\label{samortoy}
If $a$ is in the left annihilator of $W_2(R)$ and satisfies $e_6 a=e_5 a=0$ then $a\in R e_4$
\end{Lem}
\begin{proof}
Write $a=\a(e_1+e_5) + \b(e_3 + e_6) + \g e_4$ with $\a,\b,\g\in R$, $3\a=0$. Then $0=e_6 a=2 \a e_3-\a e_6+3 \b e_4$ so $\a=0$. But $0=e_5 a=-\beta  e_3-\beta  e_6+3 \beta  e_4 t$ and consequently also $\b=0$. Thus $a=\g e_4$.
\end{proof}

In case $\tor_3(R)=0$ and $\frac{1}{2}\in R$ or $\tor_2(R)=R$ we have proved
that $\theta(e_5)=e_5+re_6$, $\theta(e_6)=t e_6$, $r,t\in R$, $t\in R^\times$.
We can apply Lemma \ref{samortoy} taking $a=\theta(e_4)$ for any $\theta\in\aut(W_2(R))$. This implies that $\theta(e_4)=s e_4$ for some invertible $s\in R$. So far, when $\tor_3(R)=0$ and either $\frac{1}{2}\in R$ or $\tor_2(R)=R$, the matrix of an automorphism of $W_2(R)$ is of the form:
\begin{equation}
\begin{pmatrix}
  \begin{matrix}
  * & \rvline & * \\
  \hline
  \begin{matrix}
  0 & 0 & 0
  \end{matrix} & \rvline & s
  \end{matrix}
  & \rvline & \begin{matrix} * & *\\
  0 & 0\end{matrix}\\
\hline
  \bigzero & \rvline &
  \begin{matrix}
  1 & r \\
  0 & t
  \end{matrix}
\end{pmatrix},
\end{equation}
with $r,s,t\in R$, $s,t\in R^\times$.

We now investigate the image of $e_3$ under automorphisms of $W_2(R)$.
\begin{Lem}\label{edrem}
Assume $x\in W_2(R)$ satisfies $x e_4=x e_6=0$, $x e_5=t e_6$ ($t\in R^\times$) and $e_6 x\in R e_4$.
Then if $\tor_3(R)=0$ we have $x\in Re_3+Re_4+Re_6$.
\end{Lem}
\begin{proof}
If $x=\sum_{i=1}^6 x_i e_i$ then from $x e_4=0$ we get $x_2=3 x_1=0$.
But $x e_6=0$ gives $x_2=0$ as before and $x_5=x_1$. On the other hand
$x e_5=t e_6$ gives $-x_1 - 2 x_5=0$,
$ x_3- x_6 - t=0$. So $3x_5=0$ and $x_3-x_6\in R^\times$. Also
$e_6 x=2x_1 e_3 -x_1e_6 +3 x_3e_4$ giving $x_1=0$. Thus $x=x_3e_3+x_4e_4+x_6e_6$
\end{proof}

\begin{Lem}\label{euqelad}
Assume $\tor_3(R)=0$ and let $u=\sum_{i=1}^6\l_i e_i\in W_2(R)$ be such that $u^2=0$, $\ll_1(u)=0$, $u(e_5+r e_6)=0$, $t u e_6=-(e_5+r e_6)$ then
\begin{enumerate}
    \item  $\lambda_2\in R^\times$ and $u=\frac{\lambda _1^3}{\lambda _2^2}e_4 +\frac{\lambda_1^2}{\lambda _2}e_3+\lambda_1 e_1 +
   \lambda_2 e_2$.
   \item $r=-\l_1 t$.
   \item $\ll_1(u(x_3 e_3+x_4 e_4+x_6 e_6))=-9\l_2 x_6$.
\end{enumerate}
\end{Lem}
\begin{proof}
Since $0=\ll_1(u)=3\l_5$ we have $\l_5=0$. The fact that $u(e_5+e_6)=0$ gives $\l_1+r \l_2=0=r\l_1+\l_3-\l_6$. Since $t u e_6=-(e_5+r e_6)$ we get the equalities $t\l_2=1$ and $t\l_1=-r$. Thus $\l_2$ is invertible. Also since $u^2=0$ we have $\l_2\l_6=0$ which implies $\l_6=0$. Then $u=\sum_{i=1}^4\l_i e_i$ and $\ell_1(u (x_3e_3+x_4e_4+x_6e_6))=-9\l_2 x_6$ is readily checked. Now imposing the condition that $u^2=0$ gives
$\l_3=\l_1^2/\l_2$ and $\l_4=\l_1^3/\l_2^2$.
\end{proof}

Assume again $\tor_3(R)=0$, $\frac{1}{2}\in R$ and $\theta\in\aut_R(W_2(R))$. We know that $\theta(e_6)=t e_6$, $\theta(e_5)=e_5+r e_6$, $\theta(e_4)=s e_4$.
Applying Lemma \ref{edrem} we know that
$\theta(e_3)=x_3e_3+x_4e_4+x_6e_6$ for some $x_i\in R$.
Applying Lemma \ref{euqelad} we have $\theta(e_2)=\frac{\lambda _1^3}{\lambda _2^2}e_4 +\frac{\lambda_1^2}{\lambda _2}e_3+\lambda_1 e_1 +
   \lambda_2 e_2$ for suitable $\l_i\in R$ with $\l_2$ invertible.
   Furthermore, since $e_2e_3=2 e_1$ we have $\theta(e_2)\theta(e_3)=2\theta(e_1)$ hence \begin{center}$-9\l_2x_6=\ell_1(\theta(e_2)\theta(e_3))=2\ell_1(\theta(e_1))=2\ell_1(e_1)=0$.\end{center}
   Since $\l_2$ is invertible we get $x_6=0$ (so $\theta(e_3)\in R e_3+R e_4$). Finally
 \begin{center}  $\theta(e_1)=-\theta(e_3)\theta(e_2)=\frac{3 \lambda_1^2 x_3}{\lambda _2}e_4+2 \lambda_1 x_3 e_3 +\lambda_2 x_3 e_1$.\end{center}
   In conclusion we have
   \begin{Lem}\label{oiverp}
   If $\tor_3(R)=0$ and either $\frac{1}{2}\in R$ or $\tor_2(R)=R$ the matrix of $\theta\in\aut_R(W_2(R))$ in the basis of the $e_i$'s is of the form:
   \begin{equation}\label{seisporseis}
       \begin{pmatrix} \l_2 x_3 & 0 & 2\l_1 x_3 &\frac{3 \lambda_1^2 x_3}{\lambda_2} & 0 & 0\\
       \l_1 & \l_2 & \frac{\l_1^2}{\l_2} & \frac{\l_1^3}{\l_2^2} & 0 & 0\\
       0 & 0 & x_3 & x_4 & 0 & 0\\
       0 & 0 & 0 & s & 0 & 0\\
       0 & 0 & 0 & 0 & 1 & -\l_1 t\\
       0 & 0 & 0 & 0 & 0 & t\\
       \end{pmatrix}
   \end{equation}
   \end{Lem}
The algebra $W_2(R)$ contains $\oplus_{i=1}^4 R e_i=(\Sa_2)_R$ as a subalgebra. Also $W_2(R)\supset Re_5\oplus R_6$ also as a subalgebra. Then, under the assumptions of Lemma \ref{oiverp} we know that any automorphism $\theta\in\aut_R(W_2(R))$ preserves both subalgebras
$\theta((\Sa_2)_R)=(\Sa_2)_R$ and $\theta(Re_5+Re_6)=Re_5+Re_6$. The map
$$\aut\nolimits_R(W_2(R))\to\aut((\Sa_2)_R)$$ such that
$\theta\mapsto\theta\vert_{(\Sa_2)_R}$ is a group homomorphism. In fact it is a monomorphism because in
case $\theta\vert_{(\Sa_2)_R}=1$ we have $\l_1=0$ (see equation \eqref{seisporseis}). Thus if $\theta$ fixes $e_1,\ldots,e_4$ then also $\theta(e_5)=e_5$. Moreover since $e_3e_5=e_6$ then $\theta$ fixes also $e_6$ whence $\theta=1$.
\begin{prop}\label{algomayor}
If $\frac{1}{3},\frac{1}{2}\in R$ or $\frac{1}{3}\in R$, $2R=0$, the map
$\aut\nolimits_R(W_2(R))\to\aut((\Sa_2)_R)$ such that
$\theta\mapsto\theta\vert_{(\Sa_2)_R}$ is a group isomorphism.
\end{prop}
\begin{proof}
It only remains to prove that the map is an epimorphism.
So if $\frac{1}{3},\frac{1}{2}\in R$, take an arbitrary $f\in\aut_R((\Sa)_R)$ whose matrix relative to the basis of the $e_i$'s is given in \eqref{tsrifeht}. Define next $\hat f\colon W_2(R)\to W_2(R)$ whose restriction to $(\Sa_2)_R$ is $f$ and $f(e_5)=e_5-\frac{\m}{2} e_6$,
$f(e_6)=\frac{1}{\l}e_6$. It can be checked that
$\hat f\in\aut_R(W_2(R))$ and $\hat f\vert_{(\Sa_2)_R}=f$. Now in case  $\frac{1}{3}\in R$ and $2R=0$ take an arbitrary $f\in\aut_R((\Sa)_R)$ whose matrix relative to the basis of the $e_i$'s is given in
\eqref{dnoceseht}. Then extend $f$ to the automorphism $\hat f$ of $(W_2)_R$ such that $\hat f(e_5)=e_5+\frac{\m}{\l}e_6$ and $\hat f(e_6)=\frac{1}{\l}e_6$.
\end{proof}

 If $\ch(\F)\ne 3$ we can describe now the affine group scheme $\affaut(W_2)$. As before $\alg_{\Fs}$ will be the category of associative, commutative and unital $\F$-algebras and $\grp$ that of groups. Then $\affaut(W_2)$ is the group functor
$\affaut(W_2)\colon\alg_\Fs\to\grp$ such that $R\mapsto \aut_R((W_2)_R)$ (as usual $(W_2)_R:=W_2\otimes_\Fs R$ is the scalar extension algebra).
If $\ch(\F)\ne 2,3$ then $\frac{1}{2},\frac{1}{3}\in R$ for any $R\in\alg_\Fs$. If $\ch(\F)=2$ then $2R=0$ for any $R\in\alg_\Fs$ but $\frac{1}{3}\in R$. So in any case we can apply Proposition \ref{algomayor} to $R\in\alg_\Fs$ to compute the affine group scheme $\affaut(W_2(R))$.
Denote by $\hbox{\bf Aff}_2\colon\alg_\Fs\to\grp$ the group functor
such that $R\mapsto \Aff_2(R)=\tiny\begin{pmatrix}1 & R\cr 0 & R^\times\end{pmatrix}$. Then we claim
\begin{Th}\label{Lemma}
For a field $\F$ with $\ch(\F)\ne 3$, there is an isomorphism of group schemes $$\affaut(W_2)\cong\hbox{\bf Aff}_2.$$
\end{Th}

Any automorphism $f$ of $W_2(R)$ is of the form in \eqref{seisporseis} relative to the standard basis. We can refine its form a little. If we impose $f(e_ie_j)=f(e_i)f(e_j)$ for:
\begin{enumerate}
    \item $i=3, j=5$ we get $x_3=t$.
    \item $i=1, j=3$ we get $x_3\l_2=1$.
    \item $i=2, j=6$ we get $\l_2=1/t$.
    \item $i=6, j=3$ we get $s=t^2$.
    \item $i=5, j=3$ we get $x_4=3t^2\l_1$.
\end{enumerate}
Thus the form of a   general automorphism of $W_2(R)$ when $\tor_3(R)=0$ on canonical basis is
{\tiny \begin{equation}\label{auttor3nul}
w(x,t):=\left(
\begin{array}{cccccc}
 1 & 0 & 2 t x & 3 t^2 x^2 & 0 & 0 \\
 x & \frac{1}{t} & t x^2 & t^2 x^3 & 0 & 0 \\
 0 & 0 & t & 3 t^2 x & 0 & 0 \\
 0 & 0 & 0 & t^2 & 0 & 0 \\
 0 & 0 & 0 & 0 & 1 & -t x \\
 0 & 0 & 0 & 0 & 0 & t \\
\end{array}
\right)
\end{equation}}
So $\aut(W_2(R))=\{w(x,t)\colon x\in R, t\in R^\times\}$ and we can observe that
$w(0,1)$ is the identity and that
\begin{equation*}\label{grouplaws}
w(x,t)w(x',t')=w(x+\frac{x'}{t},tt') \hbox{ and }
w(x,t)^{-1}=w(-tx,t^{-1})
\end{equation*}
 for any $x,x',t,t'$.

\subsubsection{The case of characteristic $3$.}
In order to investigate the case of characteristic $3$ we will need to note that if $3R=0$ we have $W_2(R)^2=Re_1\oplus Re_3\oplus Re_5\oplus Re_6$.
So for $i=1,3,5,6$ and any $\theta\in\aut_R(W_2(R))$ we have
$\theta(e_i)\in Re_1\oplus Re_3\oplus Re_5\oplus Re_6$.
Now $\theta(e_5)=\l_1e_1+\l_3e_3+\l_5e_5+\l_6e_6$ and since $e_5^2=e_5$, applying $\theta$
we get $2\l_1^2+\l_1\l_5+2\l_1=0$.
Now we will use the invariant $\ell_2$ (see section \ref{guaydel}).
We have $1=\ell_2(e_5)=\ell_2(\theta(e_5))= \lambda_1^2+ \lambda_5 \lambda_1+ \lambda_5^2$ so that $ \lambda_1^2+ \lambda_5 \lambda _1+ \lambda_5^2=1$. Now from
$$\begin{cases} 2\l_1^2+\l_5\l_1+2\l_1=0\cr
\lambda _1^2+ \lambda _5 \lambda _1+ \lambda _5^2=1
\end{cases}$$
we get $2\l_1\l_5+2\l_1+\l_5^2=1$. Also from the equality $e_5^2=e_5$
applying $\theta$ we get $\l_5^2+2\l_1\l_5+2\l_5=0$ so that $\l_5=1+\l_1$. Thus we can write
\begin{equation}\label{pedocom}
\theta(e_5)=\l_1e_1+\l_3e_3+(1+\l_1)e_5+\l_6e_6.
\end{equation}
On the other hand by equation \eqref{rodaluna} we can write $\theta(e_4)=\a(e_1+e_5)+\b(e_3+e_6)+\g e_4$ and since $\theta(e_5)\theta(e_4)=0$ after an easy calculation we get $\a=\b=0$ so that $\theta(e_4)=t e_4$ (with $t\in R^\times)$. Now we put $\theta(e_2)=\sum_{i=1}^6y_ie_i$ and $\theta(e_3)=z_1e_1+z_3e_3+z_5e_5+z_6e_6$, for scalars $y_i,z_j\in R$, and given that $\theta(e_2)\theta(e_4)=\theta(e_3)$ we get $$ty_2=z_3, z_1=z_5=z_6=0$$
so that $\theta(e_3)=ty_2 e_3$. Consequently $y_2\in R^\times$.
Now writing $\theta(e_1)=x_1e_1+x_3e_3+x_5e_5+x_6e_6$, since
$\theta(e_3)\theta(e_1)=\theta(e_3)$ we get
$x_1=1$ and $x_5=0$ and so $\theta(e_1)=e_1+x_3e_3+x_6e_6$. Since $\theta(e_1)^2=2\theta(e_1)$ after expanding the corresponding equation we get $x_6=0$ so  $\theta(e_1)=e_1+x_3e_3$. Now $\theta(e_1)\theta(e_2)=0$ gives $$y_1=-x_3y_2, y_3=-x_3 y_1, y_5=0, y_6=0.$$
Thus $\theta(e_2)=-x_3y_2e_1+y_2e_2+x_3^2y_2e_3+y_4e_4$. Also $\theta(e_1)\theta(e_5)=2\theta(e_5)$ which implies
$$\lambda_1 x_3-\lambda_3=0=-\lambda_6+\lambda_1
x_3+x_3.$$
Moreover, $\theta(e_1)\theta(e_6)=\theta(e_6)$ and if we write
$\theta(e_6)=\m_1e_1+\m_3e_3+\m_5e_5+\m_6 e_6$ we get
$$\m_1=\m_5=0.$$
Also $\theta(e_2)\theta(e_3)=2\theta(e_1)$ which implies
$t=1/y_2^2$. On the other hand the equality $\theta(e_2)^2=0$ gives
$y_2(y_4+x_3^3y_2)=0$ and the invertibility of $y_2$ implies $y_4=-x_3^3y_2$.
Finally since $\theta(e_2)\theta(e_6)=2\theta(e_5)$ we get
$\l_1=y_2\m_3$ and $\m_6=\frac{\m_3y_2+1}{y_2}$. Assambling all of this together we get to the matrix of and automorphism $\theta$ of $W_2(R)$:
{\tiny $$
\left(
\begin{array}{cccccc}
 1 & 0 & x_3 & 0 & 0 & 0 \\
 -x_3 y_2 & y_2 & x_3^2 y_2 & -x_3^3
   y_2 & 0 & 0 \\
 0 & 0 & \frac{1}{y_2} & 0 & 0 & 0 \\
 0 & 0 & 0 & \frac{1}{y_2^2} & 0 & 0 \\
 \mu _3 y_2 & 0 & \mu _3 x_3 y_2 & 0 &
   \mu _3 y_2+1 & \mu _3 x_3 y_2+x_3 \\
 0 & 0 & \mu _3 & 0 & 0 & \frac{\mu _3
   y_2+1}{y_2} \\
\end{array}
\right)
$$}
whose determinant is $\frac{\left(\mu _3
   y_2+1\right){}^2}{y_2^3}$. Hence $\m_3 y_2+1\in R^\times$. So using the parameters $a=1/y_2$, $c=x_3$, $b=\m_3/a+1$ the matrix of a general automorphism $\theta$ (in the basis of the $e_i$'s) is
{\tiny\begin{equation}\label{qaesaev}
M_{a,b,c}:=\left(
\begin{array}{cccccc}
 1 & 0 & c & 0 & 0 & 0 \\
 -\frac{c}{a} & \frac{1}{a} &
   \frac{c^2}{a} & -\frac{c^3}{a} & 0 &
   0 \\
 0 & 0 & a & 0 & 0 & 0 \\
 0 & 0 & 0 & a^2 & 0 & 0 \\
 b-1 & 0 & (b-1) c & 0 & b & b c \\
 0 & 0 & a (b-1) & 0 & 0 & a b \\
\end{array}
\right)
\end{equation}}
where $a\in R^\times$, $c\in R$ and $b=y_2\m_3+1$ hence $b\in R^\times$.
In fact the set $\{M_{a,b,c}\colon a,b\in R^\times, c\in R\}$ is a group relative to matrix multiplications and obeys the rule:
$$M_{a,b,c}M_{a',b',c'}=M_{aa',bb',a'c+c'}$$
hence its identity is $M_{1,1,0}$ and also $M_{a,b,c}^{-1}=M_{a^{-1},b^{-1},-ca^{-1}}$. Thus
\begin{equation}
\aut\nolimits_R(W_2(R))\cong\{M_{a,b,c}\colon a,b\in R^\times,c\in R\}
\end{equation}
Modulo the above identification we can see that the subset $\{M_{a,1,c}\colon a\in R^\times, c\in R\}$ is a normal subgroup of $\aut_R(W_2(R))$ and it is isomorphic to $\Aff_2(R)$. Of course the quotient group is isomorphic to the multiplicative group: $$\aut(W_2(R))/\Aff\nolimits_2(R)\cong R^\times.$$
Let us compute the center $Z(\aut_R(W_2(R)))$, in we consider the equality $M_{a,b,c}M_{x,y,z}=M_{x,y,z}M_{a,b,c}$ for a fixed triple $(a,b,c)\in R^\times\times R^\times\times R$ and an arbitrary one $(x,y,z)\in R^\times\times R^\times\times R$, we find that $cx+z=az+c$ hence taking $x=1$ we get $z=az$ for any $z\in R$. So $a=1$ and this implies $cx=c$ for any  $x\in R^\times$. Thus $c(-x)=c$ also. Consequently $2c=0$ and since $3R=0$ this implies $c=0$. Then $$Z(\aut\nolimits_R(W_2(R)))=\{M_{1,b,0}\colon b\in R^\times\}\cong R^\times.$$
Furthermore we have a decomposition $M_{x,y,z}=M_{1,y,0}M_{x,1,z}$ for any $M_{x,y,z}$. So we have an isomorphism $$\aut\nolimits_R(W_2(R))\cong R^\times\times \Aff\nolimits_2(R)$$
$$M_{x,y,z}\mapsto \left(y,\begin{pmatrix}1 & z\cr 0 & x\end{pmatrix}\right)$$
induced by the decomposition of $\aut_R(W_2(R))$ as a direct product of groups. Summarizing the previous results we have:
\begin{Th}\label{tramart}
If $3R=0$ we have an isomorphism $\aut_R(W_2(R))\cong R^\times\times\Aff_2(R)$
where $Z(\aut_R(W_2(R))$ corresponds with the first factor $R^\times$. Modulo the above identification, both subgroups $R^\times$ and $\Aff_2(R)$ are normal subgroups. The general form of an element $M_{a,b,c}\in\aut_R(W_2(R))$ is in equation \eqref{qaesaev} relative to the basis of the $e_i$'s.
\end{Th}

\subsection{Conservative algebra $W(2)$}

A multiplication on the $2$-dimensional vector space $V_2$ is defined by a $2\times 2\times 2$ matrix. Their classification was given in many papers (see, for example,~\cite{kayvo2,GR11}).
Let us consider the space $W(2)$ of all multiplications on the 2-dimensional space $V_2$ with a basis $v_1,v_2$.  The definition of the multiplication $\cdot$  on the algebra $W(2)$ can be found in Introduction (see, also \cite{Kantor90,kaylopo,kayvo}). Namely, we fix the vector $v_1 \in V_2$ and define
$$(A \cdot B)(x,y)=A(v_1,B(x,y))-B(A(v_1,x),y)-B(x,A(v_1,y))$$
for $x,y \in V_2$ and $A,B \in W(2)$. The algebra $W(2)$ is conservative \cite{Kantor90}.

Let us consider the multiplications $\alpha_{ij}^k$ ($i,j,k=1,2$) on $V_2$ defined by the formula $\alpha_{ij}^k(v_t,v_l)=\delta_{it}\delta_{jl} v_k$ for all $t,l$. It is easy to see that $\{ \alpha_{ij}^k | i,j,k=1,2 \}$ is a basis of the algebra $W(2)$. The multiplication table of $W(2)$ in this basis is given in \cite{kaylopo}. In this work we use another basis for the algebra $W(2)$. Let introduce the notation
\begingroup\small\begin{longtable}{llll}
$e_1=\alpha_{11}^1-\alpha_{12}^2-\alpha_{21}^2,$&
$e_2=\alpha_{11}^2,$&
$e_3=\alpha_{22}^2-\alpha_{12}^1-\alpha_{21}^1,$&
$e_4=\alpha_{22}^1,$ \\
$e_5=2\alpha_{11}^1+\alpha_{12}^2+\alpha_{21}^2,$&
$e_6=2\alpha_{22}^2+\alpha_{12}^1+\alpha_{21}^1,$&
$e_7=\alpha_{12}^1-\alpha_{21}^1,$&
$e_8=\alpha_{12}^2-\alpha_{21}^2.$
\end{longtable}\endgroup{}

It is easy to see that the multiplication table of $W(2)$ in the basis $e_1,\dots,e_8$ is the one in following figure:

\begin{center}
\begin{longtable}{c|c|c|c|c|c|c|c|c|}
      & $e_1$ & $e_2$ & $e_3$ & $e_4$ & $e_5$ & $e_6$ & $e_7$ & $e_8$ \\ \hline
$e_1$ & $-e_1$ & $-3e_2$ & $e_3$ & $3e_4$ & $-e_5$ & $e_6$ & $e_7$ & $-e_8$ \\ \hline
$e_2$ & $3e_2$ & $0$ & $2e_1$ & $e_3$ & $0$ & $-e_5$ & $e_8$ & $0$ \\ \hline
$e_3$ & $-2e_3$ & $-e_1$ & $-3e_4$ & $0$ & $e_6$ & $0$ & $0$ & $-e_7$ \\ \hline
$e_4$ & $0$ & $0$ & $0$ & $0$ & $0$ & $0$ & $0$ & $0$ \\ \hline
$e_5$ & $-2e_1$ & $-3e_2$ & $-e_3$ & $0$ & $-2e_5$ & $-e_6$ & $-e_7$ & $-2e_8$ \\ \hline
$e_6$ & $2e_3$ & $e_1$ & $3e_4$ & $0$ & $-e_6$ & $0$ & $0$ & $e_7$ \\ \hline
$e_7$ & $2e_3$ & $e_1$ & $3e_4$ & $0$ & $-e_6$ & $0$ & $0$ & $e_7$ \\ \hline
$e_8$ & $0$ & $e_2$ & $-e_3$ & $-2e_4$ & $0$ & $-e_6$ & $-e_7$ & $0$ \\ \hline
\end{longtable}
\end{center}
\medskip
The subalgebra spanned by the elements $e_1, \ldots, e_6$ is the conservative (and, moreover,  terminal) algebra $W_2$ of commutative 2-dimensional algebras.
The subalgebra spanned by the elements $e_1, \ldots, e_4$ is the conservative (and, moreover,  terminal) algebra $S_2$ of all commutative 2-dimensional algebras with trace zero multiplication \cite{kaylopo}.

We now investigate the structure of $W(2)$ over fields of arbitrary characteristic. Regardless of $\hbox{char}(\F)$, the diagram of $W(2)$ in the basis of the $e_i$'s is:

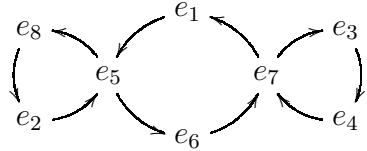
\[
\xygraph{ !{<0cm,0cm>;<1.5cm,0cm>:<0cm,1.2cm>::}
!{(0,.7) }*+{e_1}="a"
!{(0,-.7) }*+{e_6}="b"
!{(.7,0) }*+{e_7}="c"
!{(-.7,0) }*+{e_5}="d"
!{(1.4,0.5) }*+{e_3}="e"
!{(1.4,-0.5) }*+{e_4}="f"
!{(-1.4,0.5) }*+{e_8}="g"
!{(-1.4,-0.5) }*+{e_2}="h"
"a":@/_/"d" "d":@/_/"b" "b":@/_/"c" "c":@/_/"a"
"c":@/^/"e" "e":@/^/"f" "f":@/^/"c"
"d":@/_/"g" "g":@/_/"h" "h":@/_/"d"
}
\]
which is transitive (take into account that $-e_3^2+((e_3e_2)e_8)e_4=e_4$).
Now we claim:

\begin{Th}\label{norbacne}
For $\hbox{char}(\F)\ne 2,3$ we have $\M(W(2))={\rm End}_\F(W(2)),$
hence $W(2)$ is simple.
 In the characteristic $2$ case, $W(2)$ has a two dimensional ideal
 $I=\F(e_5+e_8)\oplus\F(e_6+e_7)$ and $W(2)/I \cong W_2.$
 Moreover, $\M=\M(W(2))$ has a $12$-dimensional radical $\rad(\M)$ of square zero and $\M/\rad(\M) \cong M_6(\F)\oplus M_2(\F)$.
 {  In the characteristic $3$ case  $\M=\M(W(2))$ has a $12$-dimensional radical $\M^\bot$ of square zero and $\M^\bot W(2)=\F(e_1+e_5)\oplus\F(e_3+e_6)$ is an ideal of $W(2)$}.
\end{Th}
\begin{proof}
\begin{enumerate}
    \item[(A)] Assume first that $\hbox{char}(\F)\ne 2,3$. Then we have:
\begin{longtable}{rclll}
$L_{e_8}L_{e_7}L_{e_8}$& $=$ & $6 E_{34}$ & \hbox{ implying } & $E_{3k}\in\M \hbox{ for any } k.$\\
$R_{e_7}^2$ &$ =$&$  -E_{27}$ & \hbox{ implying } & $E_{2k}\in\M \hbox{ for any } k.$\\
$L_{e_7}^2$ &$ =$&  $6E_{14}+2E_{23}$ & \hbox{ implying } & $E_{1k}\in\M \hbox{ for any } k.$\\
$L_{e_2}^2$ & $=$ & $6E_{32}+2E_{41}$ & \hbox{ implying } & $E_{4k}\in\M
\hbox{ for any } k.$\\
$L_{e_2}R_{e_4}$ & $=$ & $3E_{13}+2E_{21}-2E_{83}$ & \hbox{ implying } & $E_{8k}\in\M \hbox{ for any } k.$\\
$L_{e_2}R_{e_6}$ &$ =$ & $-E_{15}+E_{55}+E_{85}$ & \hbox{ implying } & $E_{5k}\in\M \hbox{ for any } k.$\\
$L_{e_2}$ & $=$ & $3E_{12}+2E_{31}+E_{43}-E_{65}+E_{78}$ & \hbox{ hence } & $-E_{65}+E_{78}\in\M.$
\end{longtable}
Thus $E_{7k}=(-E_{65}+E_{78})E_{8k}\in\M$ for any $k$. Therefore $E_{65}\in\M$ and so every $E_{6k}\in\M$. Thus we conclude
$\M(W(2))=\Endo_\F(W(2))$ in the case of characteristic $\ne 2,3$.

  \item[(B)] Assume now $\hbox{char}(\F)=2$.
A simple but tedious computation reveals that the radical $R=\rad(\esc{\cdot,\cdot})$ of the trace form $\esc{\cdot,\cdot}\colon W(2)\times W(2)\to \F$ given as before by $\esc{x,y}:=\text{trace}(xy)$ has
a basis given by
\begin{longtable}{lll}
$\begin{array}{lll}
r_1&=&  E_{15}+E_{18}, \\
r_2&=&  E_{16}+E_{17},\\
r_3&=&  E_{25}+E_{28},\\
r_4&=&  E_{26}+E_{27},
\end{array}$ &
$\begin{array}{lll}
r_5&=&  E_{35}+E_{38}, \\
r_6&=&  E_{36}+E_{37}, \\
r_7&=&  E_{45}+E_{48},\\
r_8&=&  E_{46}+E_{47},
\end{array}$ &
$\begin{array}{lll}
r_9&=&     E_{55}+E_{58}+E_{85}+E_{88}, \\
r_{10}&=&  E_{56}+E_{57}+E_{86}+E_{87}, \\
r_{11}&=&  E_{65}+E_{68}+E_{75}+E_{78}, \\
r_{12}&=&  E_{66}+E_{67}+E_{76}+E_{77}.
\end{array}$
\end{longtable}

It is also straightforward that $R^2=0$. The natural action
$\M\times W(2)\to W(2)$ provides the two-dimensional ideal $R\cdot W(2)\triangleleft W(2)$ which is $R\cdot W(2)=\F(e_5+e_8)\oplus\F(e_6+e_7)$. Furthermore the quotient algebra $W(2)/R W(2)$ is isomorphic to the six-dimensional algebra $B$ of section \ref{zoro}. By Theorem \ref{siri},
$W(2)/R W(2)$ is simple. One can easily check that $RW(2) \subset \text{Lann}(W(2))$ but $RW(2)\not\subset\text{Rann}(W(2))$. On the other hand, the two-sided annihilator of the ideal $RW(2)$, that is, the vector space of elements $x\in W(2)$ such that $x (RW(2))=0=(RW(2))x$ is generated by $e_3+e_7$, $e_4$, $e_5+e_8$ and $e_6+e_7$.
This implies that there is no ideal $I$ complementing $RW(2)$ (because if $I$ existed it would have dimension $6$ and  it would be contained in the linear span
 of $\{e_3+e_7, e_4, e_5+e_8, e_6+e_7\}$ which is impossible).
 Now, since $RW(2)$ is $\M$-invariant, the natural representation $\M\to\Endo(W(2))$ induces the homomorphism
 $$\M\to\Endo[W(2)/RW(2)]\times\Endo(RW(2)),\quad T\mapsto (\bar T, T\vert_{RW(2)}),$$
 where $\bar T$ is the map induced in the quotient $W(2)/RW(2)$.
 Its kernel contains $R$, since $R\,W(2)=RW(2)$ and $R\,(RW(2))=R^2W(2)=0$. A direct computation shows that $\dim\M=52$ and that this homomorphism is surjective (note that $\M(W(2)/RW(2))=\Endo[W(2)/RW(2)]$ by Theorem~\ref{siri}). Hence its kernel is the $12$-dimensional ideal $R$, and $\M/R\cong M_6(\F)\oplus M_2(\F)$ is semisimple. Since $R$ is a nilpotent ideal, we obtain $\rad(\M)=R$ (see Remark~\ref{fooo}).

{  \item[(C)]
In case $\text{char}(\F)=3$ we have

\begin{longtable}{ll}
$L_{e_1}= -E_{11}+E_{33}-E_{55}+E_{66}+E_{77}-E_{88}$, &
$R_{e_1}= -E_{11}+E_{33}+E_{51}-E_{63}-E_{73}$,\\
$L_{e_2}= -E_{31}+E_{43}-E_{65}+E_{78}$, &
$R_{e_2}=-E_{31}+E_{61}+E_{71}+E_{82},$\\
$L_{e_3}= E_{13}-E_{21}+E_{56}-E_{87}$, &
$R_{e_3}= E_{13}-E_{21}-E_{53}-E_{83}$,\\
$L_{e_4}= 0$, &
$R_{e_4}= E_{23}+E_{84}$,\\
$L_{e_5}= E_{11}-E_{33}+E_{55}-E_{66}-E_{77}+E_{88}$, &
$R_{e_5}= -E_{15}+E_{36}+E_{55}-E_{66}-E_{76}$,\\
$L_{e_6}= -E_{13}+E_{21}-E_{56}+E_{87}$, &
$R_{e_6}= E_{16}-E_{25}-E_{56}-E_{86},$\\

$L_{e_7}= L_{e_6}$, &
$R_{e_7}= E_{17}+E_{28}-E_{57}-E_{87}$,\\
$L_{e_8}= E_{22}-E_{33}+E_{44}-E_{66}-E_{77}$, &
$R_{e_8}= -E_{18}-E_{37}+E_{58}+E_{67}+E_{77},$\\

\end{longtable}
A basis for $\M$ is given by the set of matrices:
\begin{longtable}{ll}
$E_{ij}$\ for  $i=2,4,7,8$ and any $j$,\\
$E_{ai}+E_{bj}$\ for $(a,b)=(1,5),(3,6),(5,5)$ or $(6,6)$ and $(i,j)=(1,5)$ or $(3,6)$,\\
$E_{ai}-E_{bi}$\ for $(a,b)=(1,5)$ or $(3,6)$ and $i=2,4,5,6,7,8$.
\end{longtable}
We have computed again the radical $\M^\bot$ of its trace form
$\esc{f,g}:=\tr(fg)$ and it is $12$-dimensional. A basis for $\M^\bot$ is
\begin{longtable}{lll}
$E_{2,1}+E_{2,5},$&$ E_{2,3}+E_{2,6},$& $E_{1,1}+E_{1,5}-E_{5,1}+2 E_{5,5},$\\
$ E_{4,1}+E_{4,5},$&$ E_{4,3}+E_{4,6},$ &$ E_{1,3}+E_{1,6}-E_{5,3}+2 E_{5,6},$\\

$E_{7,1}+E_{7,5},$&$E_{7,3}+E_{7,6},$& $
E_{3,1}+E_{3,5}-E_{6,1}+2 E_{6,5},$\\

$E_{8,1}+E_{8,5},$&$E_{8,3}+E_{8,6},$ & $E_{3,3}+E_{3,6}-E_{6,3}+2 E_{6,6}.$

\end{longtable}
One can easily check that $(\M^\bot)^2=0$. Furthermore $\M^\bot\cdot W(2)=\F(e_1+e_5)\oplus\F(e_3+e_6)$ is a (two-sided) ideal of $W(2)$.}

\end{enumerate}
\end{proof}

\subsubsection{Automorphisms of $W(2)$.}
In this section we work again over a commutative ring $R$ and denote $W(2)_R$ the $R$-algebra $\oplus_{i=1}^8 R e_i$ where the multiplication table of the $e_i$'s is that of the multiplication table of $W(2).$
If we take a generic element $w=\sum_{i=1}^8\l_i e_i\in W(2)_R$ and compute the matrix of $L_w$
relative to the basis of the $e_i$'s we obtain:
$$
\left(\tiny
\begin{array}{cccccccc}
 -\lambda _1-2 \lambda _5 & 3 \lambda _2 & 2 \xi_1 & 0 & 0 & 0 & 0 &
   0 \\
 \xi_1 & -3 \lambda _1-3 \lambda _5+\lambda _8 & 0 & 0 & 0 & 0 & 0 & 0
   \\
 2 \lambda _2 & 0 & \xi_2 & 3\xi_1  & 0 &
   0 & 0 & 0 \\
 0 & 0 & \lambda _2 & 3 \lambda _1-2 \lambda _8 & 0 & 0 & 0 & 0 \\
 0 & 0 & 0 & 0 & -\lambda _1-2 \lambda _5 & -\xi_1 & 0 & 0 \\
 0 & 0 & 0 & 0 & -\lambda _2 & \xi_2 & 0 & 0 \\
 0 & 0 & 0 & 0 & 0 & 0 & \xi_2 & \lambda _2 \\
 0 & 0 & 0 & 0 & 0 & 0 & \xi_1 & -\lambda _1-2 \lambda _5 \\
\end{array}
\right),$$
where
$\xi_1=-\lambda _3+\lambda _6+\lambda _7$ and
$\xi_2=\lambda _1-\lambda _5-\lambda _8,$
whose characteristic polynomial is invariant under automorphism so that $L_w$ and $L_{\theta(w)}$ have the same characteristic polynomial for any automorphism
$\theta$ of $W(2)_R$. We list some of the coefficients
of that characteristic polynomial:
{\tiny
\begin{longtable}{lcl}
$\Lambda_1(w)$& $:=$ & $4 \left(3 \lambda _5+\lambda _8\right)$,\\
$\Lambda_2(w)$& $:=$ & $-4 \left(3 \lambda _1^2+3 \lambda _5 \lambda _1-3 \lambda _8 \lambda _1-15 \lambda _5^2-\lambda _8^2-3 \lambda _2 \lambda _3+3 \lambda _2 \lambda _6+3 \lambda _2 \lambda _7-12 \lambda _5 \lambda _8\right),$\\
$\Lambda_3(w)$ &$:=$ & $-2 \left(3 \lambda _5+\lambda _8\right) \left(18 \lambda _1^2+18 \lambda _5 \lambda _1-18 \lambda _8 \lambda _1-27 \lambda _5^2+\lambda _8^2-18 \lambda _2 \lambda _3+18 \lambda _2 \lambda _6+18 \lambda _2 \lambda _7-30 \lambda _5 \lambda _8\right).$
\end{longtable}}
The left annihilator of $W(2)_R$ is
\begin{center}$Re_4\oplus R(2 e_1-e_5+3 e_8)\oplus R(e_3+e_6)\oplus R(e_3+e_7).$
\end{center}

\begin{Lem}\label{monkey} Assume $W\ne 0$ to be a free $R$-module $W=\oplus_{i=1}^n Re_i$ and $M$ a submodule
with a basis $\{u_1,\ldots,u_k\}$ which is a subset of another
 basis $\{u_1,\ldots,u_n\}$ of $W$. Then if $M\subset \oplus_{i=1}^n\mi_i e_i$
for some maximal ideals $\mi_i\triangleleft R$
we have $M=0$.
\end{Lem}
\begin{proof} For $i=1,\ldots,n$ define the $R$-algebras $K_i:=R/\mi_i$ (which are
fields) and  $S=\otimes_{i=1}^n K_i$. If $M\ne 0$ the $S$-module $M\otimes_R S$ is free
with a basis of cardinal $k$ but for any $z\in M$ we have $z=\sum_{i=1}^n m_i e_i$ with
$m_i\in\m_i$ and any element $z\otimes 1_S\in M\otimes_R S$ satisfies
\begin{center}
$z\otimes 1_S=\sum_{i=1}^n m_i e_i\otimes 1_S=\sum_{i=1}^n e_i\otimes m_i 1_S$\end{center} but
$m_i 1_S=m_i (1_1\otimes\cdots\otimes 1_n)=(1_1\otimes\cdots\otimes m_i1_i\otimes\cdots \otimes 1_n)=(1_1\otimes\cdots\otimes 0\otimes\cdots \otimes 1_n)=0$.\end{proof}
The fact the $\mi_i$ is maximal in Lemma \ref{monkey} is not important. What it is essential is that it is proper (as any maximal ideal is). So we could replace the maximality hypothesis in the Lemma with $\mi_i\ne R$.

It is also easily seen that if $L$ is a free $R$-module with a (finite) basis $\{l_1,\ldots,l_n\}$ then it may not have a system of generators of cardinal $<n$. This allows to extend
the previous Lemma in the following sense:
\begin{Lem}
Assume $W\ne 0$ to be a free $R$-module $W=\oplus_{i=1}^n Re_i$ and $M$ a submodule
with a basis $\{u_1,\ldots,u_k\}$ which is a subset of another
 basis $\{u_1,\ldots,u_n\}$ of $W$. Denote by
 $p_i\colon W\to R$ the $i$-th coordinate projection relative to the basis $\{e_i\}_{i=1}^n$.
Then for each $i=1,\ldots,n$ if $p_i(R)\ne 0$ we have $p_i(W)=R$.
\end{Lem}
\begin{proof}
Assume without loss of generality that the ideal $p_1(W)$ is proper and nontrivial. We can define the ring
$$S=R/p_1(W)\otimes\overbrace{R\otimes\cdots\otimes R}^{n-1}$$
and $M_S:=M\otimes S$ is a free $S$-module of dimension $K$ but for any element $z\in M$ given by $z=\sum_{i=1}^n r_ie_i$ we have
$z\otimes 1_S=r_1e_1\otimes 1_S+\sum_{i>1}r_ie_i\otimes 1_S=\sum_{i>1}e_i\otimes r_i$ so that $\{e_i\otimes 1_S\}_{i>1}$ is a system of generators of $M_S$ of cardinal $<n$.
\end{proof}

 It is known that any commutative ring $R$
 satisfies the strong rank condition \cite[(1.38) Corollary, p. 15]{lam}, equivalently, for any
 monomorphism $R^m\to R^n$ we have $m \le n$. In particular consider the free $R$-module $R^n$ with
 canonical basis $\{e_i\}_{i=1}^n$. If a free $R$-submodule $M$ of $R^n$ has $\dim_R(M)=k$ then $k<n$. Moreover if
  $\{e_1,\ldots,e_k\}\subset M$ then we want also to prove that \begin{equation}\label{equino}M=\oplus_{i=1}^k R e_i.\end{equation}
 Indeed, take a basis $\{u_i\}_{i=1}^k$ of $M$. Then for $1\le i\le k$ we have $e_i=\sum_{q=1}^k a_i^q u_q$ and for any $q$ we also have $u_q=\sum_{j=1}^n b_q^j e_j$ (where $a_i^q,b_q^j\in R$). Thus
 $1=\sum_{q=1}^ka_i^qb_q^j=\delta_i^j$ (Kronecker's delta) or equivalently $AB=1_k$ (identity matrix $k\times k$ in $M_n(R)$) where
 $A=(a_i^j)_{i,j=1}^k$ and $B=(b_i^j)_{i,j=1}^k$.
 But since $R$ is a commutative ring, it is  stably finite (see \cite[(1.12) Proposition]{lam} and definition \cite[\S 1B, p.5]{lam}).
 So $BA=1_K$ also. Now denoting $\text{\bf u}:=(u_1,\ldots,u_k)$ and $\text{\bf e}:=(e_1,\ldots,e_k)$ we can write $A\text{\bf u}^t=\text{\bf e}^t$ hence
 $\text{\bf u}^t=B\text{\bf e}^t$ proving formula \eqref{equino}.
 \begin{Lem}\label{lema30}
 Assume $\frac{1}{2},\frac{1}{3}\in R$ and that $I$ is left ideal of $W(2)_R$ which is a free $R$-submodule and $\dim_R(I)=4$.
 Denote by $p_i\colon W(2)_R\to R$ the $i$th coordinate function relative to the basis
 $\{e_i\}$. If $p_i(x)\in R^\times$ for some $x\in I$ and $i\in\{1,2,3,4\}$ then
 $I=\oplus_{j=1}^4 R e_j$.
 \end{Lem}
 \begin{proof}
 First we prove that if some $e_i\in I$ (with $i\in\{1,2,3,4\}$) then $I=\oplus_{j=1}^4 R_j$. Assume first that $e_1\in I$, then in the first column of
 the table of multiplication of $W(2)$  we can see that $e_2,e_3\in I$ and since
 $e_4$ appears in third column we conclude $e_4\in I$. So $\oplus_{i=1}^4 Re_i\subset I$ and $\dim_R(I)=4$ implies by formula \eqref{equino} that $I=\oplus_{i=1}^4 Re_i$.
 In case $e_2\in I$ we can see that $e_1\in I$ (second column of the table of multiplication of $W(2)$). The same applies if $e_3\in I$. Finally if $e_4\in I$ then $e_3\in I$ for a similar reason. Now assume that some $x\in I$ has $p_1(x)\in R^\times$.
 We can assume without loss of generality that $p_1(x)=1$. Since we have
 $e_8 \left(e_2 \left(e_5 \left(e_1
 x\right)\right)\right)=6e_2$ then $e_2\in I$ and we apply the proved part of
 the Lemma. If some $x\in I$ has $p_2(x)\in R^\times$ again we can assume
 $p_2(x)=1$ and then since we have $e_3 \left(e_3 \left(e_3
 \left(e_1 x\right)\right)\right)=18e_4$ we conclude $e_4\in I$ and can
 apply again the proved part of the Lemma. In case $p_3(x)\in R^\times$ for some
 $x\in I$ we take into account that $e_8 \left(e_2 \left(e_2
 x\right)\right)=6e_2$ implying $e_2\in I$ (as before assuming $p_3(x)=1$). Finally if $p_4(x)\in R^\times$ for some $x\in I$ we use
 $e_2(e_2(e_2 x))=6e_2$.
 \end{proof}

 \begin{corollary}
 Assume $\frac{1}{2},\frac{1}{3}\in R$ and that $I$ is left ideal of $W(2)_R$ which is a free $R$-submodule and $\dim_R(I)=4$.
 Denote as before by $p_i\colon W(2)_R\to R$ the $i$th coordinate function relative to the basis $\{e_i\}_{i=1}^8$. The either $I=\oplus_{j=1}^4 R e_j$ or $I=\oplus_{j=5}^8 R e_j$.
 \end{corollary}
 \begin{proof}
 By Lemma \ref{lema30} either $I=\oplus_{i=1}^4 Re_i$ or $I\subset\oplus_{i=5}^8 Re_i$. But $\dim_R(I)=4$ so
 \eqref{equino} gives the equality $I=\oplus_{i=5}^8 Re_i$.
 \end{proof}
 Next we keep on assuming $\frac{1}{2},\frac{1}{3}\in R$. We want to investigate the case that $\theta\colon W(2)_R\to W(2)_R$ be an
 automorphism such that $\theta(\oplus_{i=1}^4 Re_i)=\oplus_{i=5}^8 Re_i$. Denote $\theta(e_1)=\sum_{i=5}^8\l_i e_i$. Since $\Lambda_1(e_1)=0$ we have $\Lambda_1(\theta(e_1))=4 \left(3 \lambda _5+\lambda _8\right)=0$ so $\l_8=-3\l_5$. Also $\Lambda_2(e_1)=3$ and $\Lambda_2(\theta(e_1))=12\l_5^2$ which
 implies $\l_5^2=1/4$ and in particular $\l_5\in R^\times$. Furthermore $e_1^2+e_1=0$ hence
 $\theta(e_1)^2+\theta(e_1)=0$ which (after the corresponding computation)
 gives $\l_5=\frac{1}{2}$ and $\l_7=3\l_6$. Thus we have
 \begin{equation}\label{prosf}
 \theta(e_1)=\frac{1}{2}e_5+\l_6 e_6+3\l_6e_7-\frac{3}{2}e_8.
 \end{equation}
 Next we study $\theta(e_2)=\sum_{i=5}^8\m_i e_i$. Again $\Lambda_1(e_2)=0=\Lambda_1(\theta(e_2))=4 \left(3 \mu _5+\mu _8\right)$ hence $\m_8=-3\m_5$. Moreover $\Lambda_2(e_2)=0=\Lambda_2(\theta(e_2))=12\m_5^2$ hence $\m_5^2=\m_8^2=\m_5\m_8=0$. Since
 $e_2^2=0$ we have $0=\theta(e_2)^2=\left(\mu _6-\mu _7\right) \mu _5 e_6-\left(3 \mu _6+\mu _7\right) \mu _5 e_7$. Thus $\m_6\m_5=\m_7\m_5$ and $3\m_6\m_5=\m_7\m_5$ whence $\m_6\m_5=0=\m_7\m_5$. But then $\m_5\theta(e_2)=0$, that is, $\theta(\m_5 e_2)=0$
 which gives $\m_5=0$. We get
 \begin{equation}\label{fsorp}
\theta(e_2)=\m_6 e_6+\m_7 e_7.
 \end{equation}
 But then, since $e_1e_2+3e_2=0$, applying $\theta$ and taking into account \eqref{prosf} and \eqref{fsorp}, we find $0=\theta(e_1)\theta(e_2)+3\theta(e_2)=4\mu _6 e_6 + 4\mu _7   e_7$ so that $\m_6=\m_7=0$ which is a contradiction. So far we have proved that
 no automorphism of $W(2)_R$ maps $\oplus_{i=1}^4 Re_i$ to $\oplus_{i=5}^8 Re_i$. As a consequence no automorphism of $W(2)_R$ maps $\oplus_{i=5}^8 Re_i$ to $\oplus_{i=1}^4 Re_i$.
 \begin{corollary}
 If $\frac{1}{2},\frac{1}{3}\in R$ any automorphism of $W(2)_R$ maps $\oplus_{i=1}^4 Re_i$ to itself and the same holds for $\oplus_{i=5}^8 Re_i$.
 \end{corollary}
 \begin{Lem}\label{lema30p}
 Assume $\frac{1}{2},\frac{1}{3}\in R$ and that $I$ is left ideal of $W(2)_R$ which is a free $R$-submodule of dimension $2$.
 Denote by $p_i\colon W(2)_R\to R$ the $i$th coordinate function relative to the basis
 $\{e_i\}$. Then $p_i(x)=0$ for any $x\in I$ and $i\in\{1,2,3,4\}$.
 \end{Lem}
 \begin{proof}
First we assume that some $re_i\in I$ with $r\ne 0$ and $i\in\{1,2,3,4\}$. This will take us to a contradiction. Indeed, under that assumption we have
$I_r\supset R_r\frac{e_i}{1}$ where we denote by $R_r$ the localization $RS^{-1}$ being $S=\{1,r,r^2,\ldots\}$. Let $W:=W(2)_R$ and consider the localization $W_r:=W\otimes_R R_r$ then (since $R_r$ is a flat $R$-algebra) $I_r:=I\otimes_R R_r$ is an ideal of $W_r$
which a free $R_r$-module and $\dim_{R_r}(I_r)=2$. We will identity $W_r$ with the algebra of
fractions $\frac{x}{r^n}$ ($x\in W, n\ge 0$) where $\frac{x}{r^n}=\frac{x'}{r^m}$ if and only if
$r^k(r^m x-r^n x')=0$ for some $k$.
Now if $r e_i\in I$ ($i=1,2,3,4$) then $\frac{r e_1}{1}\in I_r$ so that $\frac{e_1}{1}\in I_r$.
Consequently $R_r e_1\subset I_r$ and the multiplication table of $W$ gives
$\oplus_{i=1}^4 R_r e_i\subset I_r$ (we have identified $\frac{e_i}{1}$ with $e_i$).
But then $4\le 2$ taking dimensions. We conclude that if $r e_i\in I$ with $i\in\{1,2,3,4\}$ then $r=0$. Now consider $x\in I$ with $p_i(x)\ne 0$ and $i\in\{1,2,3,4\}$. We have $I\ni e_8 \left(e_2 \left(e_5 \left(e_1
 x\right)\right)\right)=6p_1(x)e_2$ whence $p_1(x)=0$. Next we have $I\ni e_3 \left(e_3 \left(e_3
 \left(e_1 x\right)\right)\right)=18p_2(x)e_4$ hence $p_2(x)=0$. Then  $I\ni e_8 \left(e_2 \left(e_2 x\right)\right)=6p_3(x)e_2$ whence $p_3(x)=0$ and finally the equality
 $I\ni e_2(e_2(e_2 x))=6p_4(x)e_2$ to deduce that $p_4(x)=0$.
 \end{proof}

 Consider now an automorphism $\theta$ of $W(2)_R$ (again $\frac{1}{2},\frac{1}{3}\in R$) and let us study the image $\theta(e_6)$. Since $Re_5\oplus Re_6$ is a left ideal of $W(2)_R$ and it is under the hypothesis of Lemma \ref{lema30p}, we have
 $\theta(e_5),\theta(e_6)\in \oplus_{i=5}^8 Re_i$. So for instance
 $\theta(e_5)=\sum_{i=5}^8\m_i e_i$ and
 $\theta(e_6)=\sum_{i=5}^8\l_i e_i$ and we can use again the invariants $\Lambda_1$ and
$\Lambda_2$. We have
\begin{center}$12=\Lambda_1(e_5)=\Lambda_1(\theta(e_5))=4\left(3 \mu _5+\mu _8\right)$
\end{center} whence $\m_8=3-3\m_5$. Also $-15=\Lambda_2(e_5)=\Lambda_2(\theta(e_5))=
12 \mu _5^2-18 \mu _5-9$. So we deduce that
$2 \mu _5^2-3 \mu _5+1=0$ implying that $\m_5$ is invertible. Furthermore, $\theta(e_5)^2+2\theta(e_5)=0$ hence the following elements of $R$ are zero:
\begin{center}
$2 \left(\mu _5-1\right) \mu _5,$ \
$\mu _5 \mu _6-\mu _6-\mu _5 \mu _7,$ \
$-3 \mu _5 \mu _6+3 \mu _6-\mu _5 \mu _7+2 \mu _7,$  \
$6 \left(\mu _5-1\right)^2.$
\end{center}
This implies that $\m_5=1$, $\m_7=0$ and $\m_8=0$. So $\theta(e_5)=e_5+\m_6 e_6$.
Now using again the invariants $\Lambda_1(e_6)=0$ and $\Lambda_2(e_6)=0$. We have $\Lambda_1(\theta(e_6))=4(3\l_5+\l_8)=0$ whence
 $\l_8=-3\l_5$. Also $\Lambda_2(\theta(e_6))=0$ from which we derive $\l_5^2=0$ and
 consequently $\l_8^2=0=\l_8\l_5$. We also have $\theta(e_6)^2=0$ which gives
 $\l_5 \l_6-\l_5 \l_7=0$, $-3 \l_5 \l_6-\l_5 \l_7=0$ and so $\l_5\l_6=\l_5\l_7=0$.
 As a consequence $\l_5\theta(e_6)=0$ which gives $\l_5=0$. So $\theta(e_6)=\l_6 e_6$
 and in summary we have
 \begin{equation}
 \begin{cases}\theta(e_5)=e_5+\m_6 e_6\cr \theta(e_6)=\l_6 e_6.\end{cases}
 \end{equation}
Now, under the same assumptions $\frac{1}{2},\frac{1}{3}\in R$ let us investigate
$\theta(e_7), \theta(e_8)$ for $\theta\in\aut(W(2)_R)$. As in the previous case
we have $\theta(e_7),\theta(e_8)\in\oplus_{i=5}^8 Re_i$.
Write $\theta(e_7)=\sum_{i=5}^8\g_i e_i$, since $\theta(e_6)\theta(e_7)=0$ we get
$\g_5=\g_8=0$ so that $\theta(e_7)=\g_6 e_6+\g_7 e_7$. Finally write $\theta(e_8)=\sum_{i=5}^8\d_i e_i$, from the equality $\Lambda_1(e_8)=\Lambda_1(\theta(e_8))$ we get $\delta _8=1-3 \delta _5$ and from $\Lambda _2\left(e_8\right)=\Lambda _2\left(\theta(e_8)\right)$ we have $\delta _5 \left(2 \delta _5-1\right)=0$.
Now the couple of identities $e_6e_8=e_7$ and $e_8e_6=-e_6$ give the equations
$$-\gamma _6-\delta _5 \lambda _6=0,\ -\gamma _7-3 \delta _5 \lambda _6+\lambda _6=0,\
2 \delta _5 \lambda _6=0$$
so that $\g_6=0$, $\g_7=\l_6$ and $\d_5=0$ (because $\l_6 \neq 0$). Then
 $\theta(e_7)=\l_6 e_7$ and $\theta(e_8)=\d_6 e_6+\d_7 e_7+e_8$ but since $\theta(e_8)^2=0$ we get $\d_6=0$ so that
 \begin{equation}\label{cabraloca}
 \begin{cases}\theta(e_7)=\l_6 e_7\cr \theta(e_8)=\d_7 e_7+e_8.\end{cases}
 \end{equation}

Thus we conclude
\begin{prop} If $\frac{1}{2},\frac{1}{3}\in R$ any automorphism $\theta$ of
$W(2)_R$ fixes any of the left ideals $\oplus_{i=1}^4 Re_i$,
$\oplus_{i=1}^6 Re_i$,
$\oplus_{i=5}^6 Re_i$ and $\oplus_{i=7}^8 Re_i$. Relative to the basis $\{e_i\}_{i=1}^8$ the matrix
of an automorphism is of the form
{\tiny \begin{equation}\label{actula}
\left(
\begin{array}{cccccc|cc}
 1 & 0 & 2 t x & 3 t^2 x^2 &  0 & 0 & 0 & 0\\
 x & \frac{1}{t} & t x^2 & t^2 x^3 & 0 & 0 & 0&0 \\
 0 & 0 & t & 3 t^2 x & 0 & 0 & 0&0\\
 0 & 0 & 0 & t^2 & 0 & 0 & 0&0\\
 0 & 0 & 0 & 0 & 1 & -t x & 0&0\\
 0 & 0 & 0 & 0 & 0 & t &0 &0\\
 \hline
 0 & 0 & 0 & 0 & 0 & 0 &t & 0\\
0 & 0 & 0 & 0  & 0 & 0 & tx & 1\\
\end{array}
\right)
\end{equation}}
\end{prop}
\begin{proof}
Since $\theta$ restricts to an automorphism of $W_2(R)=\oplus_{i=1}^6 Re_i$, the $6\times 6$
upper left block in \eqref{actula} is an in \eqref{auttor3nul}. It remains to prove that
$\theta(e_8)=tx e_7 +e_8$ but we have proved in \eqref{cabraloca} that $\theta(e_8)=\d_7 e_7+e_8$. Since $\theta(e_8)\theta(e_3)+\theta(e_3)=0$ we have
\begin{center}$0=(\d_7 e_7+e_8)(te_3+3t^2x e_4)+te_3+3t^2x e_4=3\d_7 t e_4-t e_3 -6t^2xe_4+te_3+3t^2x e_4=
3\d_7 t e_4-3t^2x e_4$\end{center} and since $t$ is invertible $\d_7= tx$.
\end{proof}

\subsubsection{The case $2R=0$.}
Note that necessarily $\frac{1}{3}\in R$.
In this case define $f_1:=e_5+e_8$ and $f_2=e_6+e_7$. Then $I:=Rf_1\oplus Rf_2$ is a $2$-dimensional (two-sided) ideal of $W(2)_R$ (see Theorem \ref{norbacne}). It has a basis $\{f_1,f_2\}$ which is a
subbasis of $\{e_1,e_2,e_3,e_4,e_5,e_6,e_6+e_7,e_5+e_8\}$ which can be seen to be a basis
of $W(2)_R$. Indeed the matrix of coordinates of these vectors relative to the basis of the $\{e_i\}$ is $\tiny\begin{pmatrix}I_6 & 0\cr M & I_2\end{pmatrix}$ where $I_6$ and $I_2$ denote the identity matrices of size $6$ and $2$ respectively and $M=\tiny\begin{pmatrix} 0 & 0 & 0 & 0 & 0 & 1\cr 0 & 0 & 0 & 0 & 1 & 0\end{pmatrix}$. It is easy to check that $M$ is invertible and agrees with its own inverse. The ideal $I$ satisfies $IW(2)_R=0$.
\begin{Lem}\label{invIuno}
Assume that $2R=0$ and $J\triangleleft W(2)_R$ is a $2$-dimensional ideal such that $JW(2)_R=0$ then
$J\subset I$.
\end{Lem}
\begin{proof}
Any $x\in J$ satisfies $xW(2)_R=0$ which implies that the elements of $J$ are of the form
\begin{center}$g=\lambda _3 e_3 +\lambda _4e_4 +\lambda _5(e_5+e_8)+\lambda _6 e_6+ \left(\lambda _3+\lambda _6\right)e_7$. \end{center}
Note that the $5$th and $8$th coordinates of $g$ (relative to the basis $\{e_i\}$) agree. So since
$e_2g\in J$ we must have $\l_6=p_5(e_2g)=p_8(e_2g)=\l_3+\l_6$ whence
$\l_3=0$ and we have proved that the elements of $J$ satisfy $p_3(J)=0$.
So a general element of $J$ is of the form
$g=\l_4e_4 +\l_5(e_5+e_8)+\l_6 (e_6+e_7)$. But
$J\ni e_2g=\l_4 e_3 + \l_6e_5 +
(\l_3 +\l_6)e_8$ which implies $\l_4=0$. Thus $g\in R(e_5+e_8)\oplus R(e_6+e_7)$.

\end{proof}

Under the hypothesis in the title of this subsection, if $\theta\in\aut(W(2)_R)$, Lemma \ref{invIuno} implies $\theta(I)\subset I$ (recall that
$I$ is the ideal $I=Rf_1\oplus Rf_2$ defined above). Consequently $I\subset \theta^{-1}(I)\subset I$ so that $\theta(I)=I$ for any $\theta\in\aut(W(2)_R)$.
Since $W(2)_R/I\cong W_2$ (see Theorem \ref{norbacne}) any $\theta\in\aut(W(2)_R)$
induces an automorphism $\bar\theta\colon W_2\to W_2$. Then the matrix of $\bar\theta$ relative to the basis $\{\bar e_i\}_{i=1}^6$ (begin $\bar e_i:=e_i+I$) is the one in formula
\eqref{auttor3nul}. So the matrix of $\theta$ relative to the basis $\{e_1,\ldots,e_6,f_1,f_2\}$ of $W(2)_R$ is of the form
{\tiny \begin{equation}\label{auttor3nuldos}
\left(
\begin{array}{cccccc|cc}
 1 & 0 & 0 &  t^2 x^2 & 0 & 0 & a_1 & a_2 \\
 x & \frac{1}{t} & t x^2 & t^2 x^3 & 0 & 0 & a_3 & a_4 \\
 0 & 0 & t &  t^2 x & 0 & 0 & a_5 & a_6\\
 0 & 0 & 0 & t^2 & 0 & 0 & a_7 & a_8\\
 0 & 0 & 0 & 0 & 1 & t x & a_9 & a_{10}\\
 0 & 0 & 0 & 0 & 0 & t & a_{11} & a_{12}\\
 \hline
 0 & 0 & 0 & 0 & 0 & 0 & a_{13} & a_{14}\\
 0 & 0 & 0 & 0 & 0 & 0 & a_{15} & a_{16}\\
\end{array}
\right)
\end{equation}}
where $t\in R^\times$. Furthermore, if we write the matrix of $\theta$ relative to the basis of
the $e_i$'s and impose the conditions for automorphism we find the relations
\begin{center}
    $a_1=a_2=a_3=a_4=a_5=a_6=a_7=a_8= 0, a_9= a_{10} t x,$\\
$a_{11}= a_{10} t, a_{12}= 0, a_{14}= 0, a_{15}= a_{13} x, a_{16}= \frac{a_{13}}{t}.$
\end{center}
\begin{Lem} In case $2R=0$ the matrix of an automorphism $\theta\in\aut(W(2)_R)$ relative to the basis $\{e_1,\ldots,e_6,f_1,f_2\}$ of $W(2)_R$ is
{\tiny \begin{equation}\label{nlhba}
\Omega_{t,x,v,u}=
\left(
\begin{array}{cccccccc}
 1 & 0 & 0 & t^2 x^2 & 0 & 0 & 0 & 0 \\
 x & \frac{1}{t} & t x^2 & t^2 x^3 & 0 & 0 & 0 & 0 \\
 0 & 0 & t & t^2 x & 0 & 0 & 0 & 0 \\
 0 & 0 & 0 & t^2 & 0 & 0 & 0 & 0 \\
 0 & 0 & 0 & 0 & 1 & t x & u t x & u \\
 0 & 0 & 0 & 0 & 0 & t & u t & 0 \\
 0 & 0 & 0 & 0 & 0 & 0 & v & 0 \\
 0 & 0 & 0 & 0 & 0 & 0 & v x & \frac{v}{t} \\
\end{array}
\right)\end{equation}}
where we have replaced $a_{10}$ with $u$ and $a_{13}$ with $v$. Furthermore $t,v\in R^\times$, $x,u\in R$.
\end{Lem}
We have the relations
\begin{center}$\Omega_{t,x,v,u}\Omega_{t',x',v',u'}=\Omega_{tt',x+x'/t,vv',u'+uv'/t'}$ and
$\Omega_{t,x,v,u}^{-1}=\Omega_{1/t,tx,1/v,tu/v}$.\end{center}
The set
$G_1:=\{\Omega_{t,x,1,0}\colon t\in R^\times, x\in R\}$ is a subgroup of $\aut(W(2)_R)$ isomorphic to $\Aff_2(R)$. Indeed if we consider
$\Aff_2(R)=\left\{\tiny\begin{pmatrix}1 & x\cr 0 & t\end{pmatrix}\colon t\in R^\times, x\in R\right\}$ we have a group isomorphism $\g_1\colon\Aff_2(R)\to G_1$ such that $\tiny\begin{pmatrix}1 & x\cr 0 & t\end{pmatrix}\mapsto\Omega_{t,t^{-1}x,1,0}$.
On the other hand  $G_2:=\{\Omega_{1,0,v,u}\colon v\in R^\times, u\in R\}$ is also a subgroup of $\aut(W(2)_R)$ isomorphic to $\Aff_2(R)$ via the map $\g_2\colon\Aff_2(R)\to G_2$ such that $\tiny\begin{pmatrix}1 & x\cr 0 & t\end{pmatrix}\mapsto\Omega_{1,0,t,x}$.
It is easily seen that $G_2$ is a normal subgroup of $\aut(W(2)_R)$ and the
map $\rho\colon G_1\to\aut(G_2)$ given by
$$\rho(\Omega_{t,x,1,0})(\Omega_{1,0,v,u})=\Omega_{t,x,1,0}\Omega_{1,0,v,u}\Omega_{t,x,1,0}^{-1}=\Omega_{1,0,v,t u}$$
is a group homomorphism.
We also have $\Omega_{t,x,v,u}=\Omega_{t,x,1,0}\Omega_{1,0,v,u}$ and
so $\Aut(W(2)_R)=G_2\ltimes G_1$ with multiplication
$$(g_2g_1)(g_2'g_1')=\left[g_2\rho(g_1)(g_2')\right](g_1g_1').$$
If we define $\tau_{x,t}:=\tiny\begin{pmatrix}1 & x\cr 0 & t\end{pmatrix}$ so that
$\Aff_2(R)=\{\tau_{x,t}\colon x\in R, t\in R^\times\}$ then we have an action of
$\Aff_2(R)$ on itself by automorphisms $\rho'\colon\Aff_2(R)\to\aut(\Aff_2(R))$ given by $\rho'(\tau_{x,t})(\tau_{u,v})=\tau_{tu,v}$. Then there is a commutative
square
\begin{center}\small
\begin{tikzcd}[column sep=small]
G_1 \arrow[r, "\rho"] \arrow[d,"\g_1"']
& \aut(G_2) \arrow[d, "\hbox{inn}\g_2" ] & \theta\arrow[d,mapsto]\\
\Aff_2(R) \arrow[r, "\rho'"' ]
& \aut(\Aff_2(R)) & \g_2^{-1}\theta\g_2
\end{tikzcd}
\end{center}
and we conclude that $\aut(W(2)_R)\cong\Aff_2(R)\ltimes\Aff_2(R)$.

\subsubsection{The case $3R=0$.}
Note that necessarily $\frac{1}{2}\in R$.
Consider an $R$-algebra $A$ which is a free $R$-module with a finite basis. Let $\M:=\M(A)$ be its multiplication algebra and $\tr\colon\M\to R$ the trace (so $\tr(T)$ is the trace of the matrix of $T$ relative to any basis of the $R$-module $A$). Also denote by $k\colon\M\times\M\to R$ the symmetric $R$-bilinear map $k(T,S):=\tr(TS)$. This satisfies
$k(TT',S)=k(T,T'S)=K(T',ST)$ for any $T,T',S\in\M$. Thus $\M^\bot:=\{T\in\M\colon k(T,\_)=0\}$ is an ideal of $\M$ and $\M^\bot A$ an ideal of $A$. There is also an action $\aut(A)\times\M\to\M$ such that $\varphi\cdot T=T^*:=\varphi T\varphi^{-1}$ for any $\varphi\in\aut(A)$ and $T\in\M$. Furthermore
$k(T^*,S^*)=K(T,S)$ for any $S,T\in\M$ so that
$(\M^\bot)^*\subset\M^\bot$ or equivalently $\aut(A)\cdot\M^\bot\subset\M^\bot$. Consequently the ideal $\M^\bot A$ of $A$ is invariant under automorphisms of $A$: for any $\varphi\in\aut(A)$, $T\in\M^\bot$ and $a\in A$ one has $\varphi(T(a))=T^*\varphi(a)\in\M^\bot A$.

\begin{Remark}\label{gfurp}\rm
Let $\F$ be an arbitrary field in this Lemma and  $U$ be a finite-dimensional $\F$-algebra, $\M=\M(U)$ its multiplication algebra, $I\triangleleft\M$ and $R\in\alg_\F$. If $j\colon I U\to U$ is the inclusion,
identifying $IU\otimes R$ with $(IU)_R$ via
$j\otimes 1_R\colon IU\otimes R\to U_R$, we have
$(IU)_R= I_R U_R$.
\end{Remark}
We now particularize considering $W(2)_R$.
We start with $W(2)$ over a field $\F$ of characteristic $3$ and take $A=W(2)_R$.
If we denote $\M=\M(W(2))$ then $\M_R$ can be identified
with $\M(W(2)_R)$ (\cite[(2.5) Lemma (a)]{finston}).
Also we have $k\colon\M\times\M\to\F$ as above:
$k(T,S)=\tr(TS)$ inducing $k_R\colon\M_R\times\M_R\to R$ and we have the standard result that
$(\M_R)^\bot\cong (\M^\bot)_R$. By
Theorem \ref{norbacne} we have $\M^\bot W(2)=\F(e_1+e_5)\oplus\F(e_3+e_6)$
hence by Remark \ref{gfurp}, $\M_R^\bot W(2)_R=R(e_1+e_5)\oplus R(e_3+e_6)$. So this ideal is
invariant under automorphisms of $W(2)_R$.

Next we compute the quotient algebra $W(2)_R/I$ where $I=R(e_1+e_5)\oplus R(e_3+e_6)$. We consider a basis
of $W(2)_R/I$ given by
$$\begin{cases}f_i=e_i+I, \hbox{ for } i=1,2,3,4,\cr
f_5=e_8+I,\cr
f_6=2e_7+I.\end{cases}$$
The multiplication of the quotient algebra relative to this basis is given in the following table
\begin{center}
\begin{longtable}{|c|c|c|c|c|c|c|}
\hline
      & $f_1$ & $f_2$ & $f_3$ & $f_4$ & $f_5$ & $f_6$ \\ \hline
$f_1$ & $2f_1$ & $0$ & $f_3$ & $0$ & $2f_5$ & $f_6$  \\ \hline
$f_2$ & $0$ & $0$ & $2f_1$ & $f_3$ & $0$ & $2f_5$ \\ \hline
$f_3$ & $f_3$ & $2f_1$ & $0$ & $0$ & $f_6$ & $0$  \\ \hline
$f_4$ & $0$ & $0$ & $0$ & $0$ & $0$ & $0$ \\ \hline
$f_5$ & $f_1$ & $0$ & $2f_3$ & $0$ & $f_5$ & $2f_6$ \\ \hline
$f_6$ & $2f_3$ & $f_1$ & $0$ & $0$ & $2e_6$ & $0$ \\ \hline
\end{longtable}
\end{center}
So we conclude that $W(2)_R/I\cong W_2(R)$ and any automorphism of $W(2)_R$ induces an automorphism
of $W_2(R)$ whose matrix relative to the basis of the
$f_i$'s is given in \eqref{qaesaev}.
Consequently any automorphism $\theta$ of $W(2)_R$ acts in the form
\begin{longtable}{lcl}
$\theta(e_1)$ &$=$&$e_1+c e_3+\left(e_1+e_5\right) t_1+\left(e_3+e_6\right) t_2=\left(t_1+1\right)e_1+ \left(c+t_2\right)e_3+t_1 e_5 +t_2 e_6,$\\

$\theta(e_2)$ & $=$&$-\frac{c }{a}e_1+\frac{1}{a}e_2+\frac{c^2}{a}e_3-\frac{c^3}{a} e_4+t_3\left(e_1+e_5\right) +t_4\left(e_3+e_6\right)=$\\

\multicolumn{3}{r}{$\left(t_3-\frac{c}{a}\right)e_1+\frac{1}{a}e_2+
\left(\frac{c^2}{a}+t_4\right)e_3 -\frac{c^3}{a}e_4+t_3e_5+t_4e_6,$}\\

$\theta(e_3)$ & $=$&$a e_3+t_5\left(e_1+e_5\right) +t_6\left(e_3+e_6\right)=\left(a+t_6\right)e_3 +t_5e_1 +t_5e_5+t_6 e_6,$\\

$\theta(e_4)$ & $=$&$a^2 e_4+t_7\left(e_1+e_5\right) +t_8\left(e_3+e_6\right)= t_7e_1+t_8 e_3 +a^2 e_4+t_7e_5 +t_8 e_6.$
\end{longtable}
But imposing the conditions $\theta(e_ie_j)=\theta(e_i)\theta(e_j)$ for $i,j\in\{1,2,3,4\}$ we get
$$t_2=c t_1, t_i=0 \text{ for } i\ge 3.$$
Thus, the coordinates of $\theta(e_i)$, $i=1,2,3,4$ relative to the $\{e_j\}_{j=1}^8$ writen in matrix form give
$$\tiny\left(
\begin{array}{cccccccc}
 1 & 0 & c & 0 & 0 & 0 & 0 & 0 \\
 -\frac{c}{a} & \frac{1}{a} & \frac{c^2}{a} & -\frac{c^3}{a} & 0 & 0 & 0 & 0 \\
 0 & 0 & a & 0 & 0 & 0 & 0 & 0 \\
 0 & 0 & 0 & a^2 & 0 & 0 & 0 & 0 \\
\end{array}
\right).$$
On the other hand since the image of $e_1+e_5$ and $e_3+e_6$
is in $R(e_1+e_5)\oplus R(e_3+e_6)$ we have

\begin{longtable}{lcl}
    $\theta(e_5)$ & $=$ &$-\theta(e_1)+x_1\left(e_1+e_5\right)+x_2
    \left(e_3+e_6\right) x_2=$\\
    &&  $\left(-c-t_2+x_2\right)e_3 +\left(-t_1+x_1-1\right)e_1 + \left(x_1-t_1\right)e_5+\left(x_2-t_2\right)e_6 ,$\\

    $\theta(e_6)$ & $=$&$-\theta(e_3)+x_3\left(e_1+e_5\right)+x_4
    \left(e_3+e_6\right) x_2=$\\
    &&  $\left(-a-t_6+x_4\right)e_3 +\left(x_3-t_5\right)e_1 + \left(x_3-t_5\right)e_5+\left(x_4-t_6\right)e_6.$
\end{longtable}
Imposing the conditions $\theta(e_ie_j)=\theta(e_i)\theta(e_j)$ for $i\in\{1,2,3,4\}$ and $j\in\{5,6\}$ we get
\begin{center}$t_1=\frac{t_6}{a},$ \  $x_2 = c x_1,$ \  $x_3=0,$  \ $x_1=\frac{x_4}{a}$\end{center}
and the coordinates of $\theta(e_i)$ with $i=1,\ldots, 6$ relative to the basis $\{e_j\}_{j=1}^8$ writen in a matrix form are
$$\tiny
\left(
\begin{array}{cccccccc}
 1 & 0 & c & 0 & 0 & 0 & 0 & 0 \\
 \frac{2 c}{a} & \frac{1}{a} & \frac{c^2}{a} & \frac{2 c^3}{a} & 0 & 0 & 0 & 0 \\
 0 & 0 & a & 0 & 0 & 0 & 0 & 0 \\
 0 & 0 & 0 & a^2 & 0 & 0 & 0 & 0 \\
 b-1 & 0 & c (b -1) & 0 & b & b c & 0 & 0 \\
 0 & 0 & a (b-1) & 0 & 0 & a b & 0 & 0 \\
\end{array}
\right).$$
Finally writing $\theta(e_7)=\sum\l_i e_i$ and $\theta(e_8)=\sum\m_i e_i$
and imposing the conditions $\theta(e_ie_j)=\theta(e_i)\theta(e_j)$ for $i\in\{1,\ldots, 8\}$ and $j\in\{7,8\}$ we get the matrix of a general automorphism $\theta\in\aut_R(W(2)_R)$ which is
{\tiny \begin{equation}\label{frogopro}
  M_{a,b,c,k}:=\left(
\begin{array}{cccccccc}
 1 & 0 & c & 0 & 0 & 0 & 0 & 0 \\
 -\frac{c}{a} & \frac{1}{a} & \frac{c^2}{a} & -\frac{c^3}{a} & 0 & 0 & 0 & 0 \\
 0 & 0 & a & 0 & 0 & 0 & 0 & 0 \\
 0 & 0 & 0 & a^2 & 0 & 0 & 0 & 0 \\
 b-1 & 0 & b c-c & 0 & b & b c & 0 & 0 \\
 0 & 0 & a b-a & 0 & 0 & a b & 0 & 0 \\
 0 & 0 & a k & 0 & 0 & a k & a & 0 \\
 -k & 0 & -c k & 0 & -k & -c k & -c & 1 \\
\end{array}
\right)
\end{equation}}
with $a,b\in R^\times$, $k,c\in R$.

\begin{Th}
If $3R=0$ the matrix of any automorphism of $W(2)_R$ relative to a basis $\{e_i\}$ with multiplication table as in the table of multiplication of $W(2)$  is of the form \eqref{frogopro}
with $a,b\in R^\times$, $c,k\in R$.
\end{Th}

We have $$M_{a,b,c,d}M_{a',b',c',d'}=M_{aa',bb',c a'+c',d b'+d},\quad M_{a,b,c,d}^{-1}=M_{\frac{1}{a},\frac{1}{b},-\frac{c}{a},-\frac{k}{b}},$$
then $\aut_R(W(2)_R)\cong\Aff_2(R)\times\Aff_2(R)$ via the isomorphism $M_{a,b,c,d}\mapsto\tiny \left(\begin{pmatrix}1 & c\cr 0 & a\end{pmatrix}, \begin{pmatrix}1 & d\cr 0 & b\end{pmatrix}\right)$
and as an affine group scheme $$\affaut(W(2))\cong \hbox{\bf Aff}_2\times \hbox{\bf Aff}_2.$$
In this case, we have a direct product dislike the case $2R=0$ in which the product was semidirect.

\appendix
\section*{Appendix A. Differences from the published version}
\addcontentsline{toc}{section}{Appendix A. Differences from the published version}

\noindent Some typos have been corrected.

\end{document}